\documentclass[a4paper,10pt,oneside]{article}
\usepackage{amssymb,amsmath,amsthm,bm,booktabs,caption,graphicx}
\usepackage[font={sf,small}]{floatrow}
\usepackage[pagebackref,colorlinks]{hyperref}
\newcommand{\tabcaption}{\def\@captype{table}\caption}
\newtheorem{lem}{Lemma}[section]
\newtheorem{thm}{Theorem}[section]
\newtheorem{propo}{Proposition}[section]
\newtheorem{rem}{Remark}[section]
\newtheorem{assum}{Assumption}[section]
\newcommand{\lam}{\lambda}
\newcommand{\verti}[1]{|#1|}
\newcommand{\vertiii}[1]{|#1|}
\newcommand{\we}{\widetilde}
\newcommand{\norm}[1]{\left\lVert#1\right\rVert}
\newcommand{\bndint}[1]{\sum_{T\in\mathcal T_h}\langle #1 \rangle_{\partial T}}
\newcommand{\bint}[2]{\langle #1 \rangle_{#2}}
\newcommand{\inpro}[1]{\langle #1 \rangle_h}
\newcommand{\ninpro}[2]{\langle #1 \rangle_{#2}}

\newcommand{\bmnabla}{\bm\nabla}
\newcommand{\bmsigma}{\bm\sigma}
\newcommand{\bmsiglam}{\bm\sigma_{\lambda}}
\numberwithin{equation}{section}

\begin{document}
\title
{
  \Large\bf BPX preconditioner for nonstandard finite element methods for diffusion problems
  \thanks{This work was supported by National Natural Science Foundation of China (11171239) and  Major Research
    Plan of National Natural Science Foundation of China (91430105).}
}

\author
{
  Binjie Li\thanks{Email: libinjiefem@yahoo.com}, \quad Xiaoping Xie \thanks{Corresponding author.
    Email: xpxie@scu.edu.cn}\\
  {School of Mathematics, Sichuan University, Chengdu 610064, China}
}

\date{}
\maketitle
\begin{abstract}
  This paper proposes and analyzes an optimal preconditioner for a general linear symmetric positive
  definite (SPD) system by following the basic idea of the well-known BPX framework. The SPD system
  arises from a large number of nonstandard finite element methods for diffusion problems,
  including the well-known hybridized Raviart-Thomas (RT) and Brezzi-Douglas-Marini (BDM) mixed element
  methods, the hybridized discontinuous Galerkin (HDG) method, the Weak Galerkin (WG) method, and the
  nonconforming Crouzeix-Raviart (CR) element method. We prove that the presented preconditioner is
  optimal, in the sense that the condition number of the preconditioned system is independent of the
  mesh size. Numerical experiments are provided to confirm the theoretical results.

\end{abstract}

\section{Introduction}
This paper is to design an efficient preconditioner for a large class of nonstandard finite element
methods for solving the diffusion model
\begin{equation}\label{eq:model}
  \left\{
    \begin{array}{rcll}
      -\text{div}(\bm A\bmnabla u)\!\! &\!\!=\!\!&\!\! f &\text{ in $\Omega$,}\\
      u\!\! &\!\!=\!\!&\!\! g &\text{ on $\partial\Omega$,}
    \end{array}
  \right.
\end{equation}
where $\Omega\subset \mathbb R^d$ ($d$ = 2, 3) is a bounded polyhedral domain, the diffusion tensor
$
\bm A:\Omega\to\mathbb R^{d\times d}
$
is a matrix function that is assumed to be symmetric and uniformly positive definite,
$f\in L^2(\Omega)$ and $g\in H^{\frac{1}{2}}(\partial\Omega)$.

Let $\mathcal T_h$ be a triangulation of $\Omega$, and $\mathcal F_h$ be the set of all faces of
$\mathcal T_h$. We introduce a finite dimensional space
\begin{equation}
  \mathbb M_{h,k}:=\{\mu_h\in L^2(\cup_{F\in\mathcal F_h}F):\mu_h|_F\in P_k(F)
  ~\text{ for all $F\in\mathcal F_h$ and $\mu_h|_{\partial\Omega}=0$}\},
\end{equation}
with $P_k(F)$ denoting the set of polynomials of degree $\leqslant k$ on $F$.  Consider
the following general symmetric and positive definite (SPD) system for equation \eqref{eq:model}:
Seek $\lam_h\in \mathbb M_{h,k}$ such that
\begin{equation}\label{eq:sys}
  d_h(\lam_h,\mu_h) = b(\mu_h)~\text{ for all $\mu_h\in \mathbb M_{h,k}$}.
\end{equation}
Here
$d_h(\cdot,\cdot):\mathbb M_{h,k}\times \mathbb M_{h,k}\to\mathbb R$ is an inner-product on
$\mathbb M_{h,k}$ and $b_h(\cdot):\mathbb M_{h,k}\to\mathbb R$ is a linear functional on
$\mathbb M_{h,k}$.

The first class of nonstandard finite element methods that fall into the framework \eqref{eq:sys}
are hybrid or hybridized finite element methods (\cite{Babuska-Oden-Lee1977, Oden-Lee1977,Raviart-Thomas1977,
  Raviart-Thomas1979, ArnoldBrezzi1985, BrezziDouglasMarini1985,COCKBURN_2004,COCKBURN_2005, super_con_LDG,
  unified_dg,unified_hdg, projection_based_hdg,Li-Xie2014}). Due to the relaxation of the constraint of
continuity at the inter-element boundaries by introducing some Lagrange multipliers, the corresponding
hybrid method allows for piecewise-independent approximation to the potential or flux solution. Thus,
after local elimination of unknowns defined in the interior of elements, the method leads to a SPD
discrete system of the form \eqref{eq:sys}, where the unknowns are only the globally coupled degrees
of freedom describing the Lagrange multiplier. In \cite{ArnoldBrezzi1985, BrezziDouglasMarini1985},
the Raviart-Thomas (RT) \cite{RT} and Brezzi-Douglas-Marini (BDM) mixed methods were shown to have equivalent
hybridized versions. A new characterization of the approximate solution of hybridized mixed methods was
developed and applied in \cite{COCKBURN_2004} to obtain an explicit formula for the entries of the matrix
equation for the Lagrange multiplier unknowns.  An overview of some new hybridization techniques was presented
in \cite{COCKBURN_2005}. In \cite{unified_hdg} a unifying framework for hybridization of finite element methods
was developed.  Error estimates of some hybridized discontinuous Galerkin (HDG) methods were derived in
\cite{super_con_LDG,projection_based_hdg,Li-Xie2014}.
\par
The weak Galerkin (WG) method  \cite{WangYe2013, Mu-Wang-Ye1,Mu-Wang-Wang-Ye} is the second class of
nonstandard approach that applies to the framework \eqref{eq:sys}. The WG method is designed by
using a weakly defined gradient operator over functions with discontinuity, and allows the use of totally
discontinuous functions in the finite element procedure. The concept of weak gradients provides a systematic
framework for dealing with discontinuous functions defined on elements and their boundaries in a near classical
sense \cite{WangYe2013}. Similar to the hybrid methods, the WG scheme can be reduced to the form \eqref{eq:sys}
after local elimination of unknowns defined in the interior of elements. We note that when $\bm A$ in
(\ref{eq:model}) is a piecewise-constant matrix,  the WG method is, by introducing the discrete weak
gradient as an independent variable, equivalent to the hybridized version of the RT or BDM mixed methods.
For the discretization of the diffusion model \eqref{eq:model} on simplicial 2D or 3D meshes, we refer
to \cite{Li;Xie;2014;WG} for a multigrid WG algorithm,  and to \cite{Chen-Wang-Wang-Ye2014} for an auxiliary
space multigrid preconditioner for the WG method as well as a reduced system of the weak Galerkin method
involving only the degrees of freedom on edges/faces.
\par
Besides, some noncnforming methods, e.g. the nonconforming Crouzeix-Raviart element method \cite{CR}, can
also lead to a SPD discrete system of the form \eqref{eq:sys}. To this end, one needs to introduce a
special projection of the flux solution to the element boundaries as the trace approximation.  We refer
to \cite{mg_p1,Braess;Verfurth1990,Brenner.S1992,ARBOGAST;Chen1995,Lazarov;2003,Rahman;2005,Kraus;2008,Wang;Chen;2012,Zhu.Y2014}
for multigrid algorithms or preconditioning for the CR or CR-related nonconforming finite element methods.
In particular, in \cite{Brenner.S1992}, an optimal-order multigrid method was proposed and analyzed for
the lowest-order Raviart-Thomas mixed element based on the equivalence between Raviart-Thomas mixed methods
and certain nonconforming methods.

\par
As far as we know, the first preconditioner for the system \eqref{eq:sys} was developed in \cite{schwarz_pre},
where a Schwarz preconditioner was designed for the hybridized RT and BDM mixed element methods.
In \cite{G_2009} a convergent V-cycle multigrid method was proposed for the hybridized mixed methods for
Poisson problems with full elliptic regularity. By following the idea of \cite{G_2009}, a non-nested multigrid
V-cycle algorithm, with a single smoothing step per level, was analyzed in \cite{multigrid_hdg} for the system
\eqref{eq:sys} arising from one type of HDG method, where   only a weak  elliptic regularity is   required. In \cite{Li;Xie;2015;HDG},
a general framework for designing fast solvers for the system \eqref{eq:sys} was presented without any regularity assumption.
\par
It is well known that the BPX multigrid framework, developed by Bramble, Pasciak and Xu \cite{BPX}, is widely
used in the analysis of multigrid and domain decomposition methods. We refer to \cite{Bramble1993,BrambleKwakPa1994,
  BramblePa1993,BrambleZhang2001,Duan;Gao;Tan;Zhang2007,GoKan2003,Wang1993, xu_2002, xu_2009,XuZhu2007,XuChen2001}
for the development and applications of the BPX framework. In \cite{Xu;1996} an abstract framework of auxiliary
space method was proposed and an optimal multigrid technique was developed for general unstructured grids. Especially,
in \cite{xu_2009} an overview of multilevel methods, such as V-cycle multigrid and BPX preconditioner, was given
for solving various partial differential equations on quasi-uniform meshes,  and the methods were extended to graded
meshes and completely unstructured grids.
\par
In this paper, we shall follow the basic ideas of  (\cite{BPX}, \cite{Xu;1996}, \cite{xu_2009}) to construct a
BPX preconditioner for the system \eqref{eq:sys}, which is, due to the definition of the discrete space
$\mathbb M_{h,k}$, corresponding to nonnested multilevel finite element spaces. We will show the proposed
preconditioner is optimal.
\par
We arrange the rest of the paper as follows. Section 2 introduces some notation and preliminaries.  Section 3 introduces and analyzes a general auxiliary space preconditioner.  Section 4 constructs
the BPX preconditioner and derives the condition number estimation of the preconditioned system. Section 5 shows some applications of the
proposed preconditioner. Finally, Section 6 provides some numerical results.
\section{ Notations and preliminaries}
Throughout this paper, we use the standard definitions of Sobolev spaces and their norms and semi-norms (cf. \cite{ADAMS}),
namely for an arbitrary open set $D\subset\mathbb R^d$ and any nonnegative integer $s$,
\begin{displaymath}
  \begin{array}{rcl}
    H^s(D) &:=& \{v\in L^2(D): \partial^{\alpha}v \in L^2(D), \forall |\alpha| \leqslant s\},\\
    \norm{v}_{s,D} &:=& (\sum_{|\alpha|\leqslant s}\int_{D}|\partial^{\alpha}v|^2)^{\frac{1}{2}},
    \quad |{v}|_{s,D} := (\sum_{|\alpha|=s}\int_{D}|\partial^{\alpha}v|^2)^{\frac{1}{2}}.
  \end{array}
\end{displaymath}
We denote respectively by $(\cdot,\cdot)_D$ and $\bint{\cdot,\cdot}{\partial D}$  the $L^2$
inner products on $L^2(D)$ and $L^2(\partial D)$, and respectively by $\norm{\cdot}_D$ and
$\norm{\cdot}_{\partial D}$ the $L^2$-norms on $L^2(D)$ and $L^2(\partial D)$.  In particular,
$(\cdot,\cdot)$ and $\norm{\cdot}$ abbreviate $(\cdot,\cdot)_{\Omega}$ and $\norm{\cdot}_{\Omega}$,
respectively.

Let $\mathcal T_h$ be a conforming shape-regular
triangulation of the  polyhedral domain $\Omega$. For any $T\in\mathcal T_h$, $h_T$ denotes the diameter of $T$, and we set
$h: = \max_{T\in\mathcal T_h}h_T$. We define the mesh-dependent inner product
$
\langle\cdot,\cdot\rangle_h:\mathbb M_{h,k}\times \mathbb M_{h,k}\to\mathbb R
$
and the norm $\norm{\cdot}_h: \mathbb M_{h,k}\to\mathbb R$ as follows:
for any $\lam_h,\mu_h\in \mathbb M_{h,k},$
\begin{equation}
  \langle\lam_h,\mu_h\rangle_h :=  \sum_{T\in\mathcal T_h}h_T\int_{\partial T}\lam_h\mu_h,
  \quad \norm{\mu_h}_h :=\langle\mu_h,\mu_h\rangle_h^{1/2}.
\end{equation}
We also need the following notation: for any $\mu\in L^2(\partial T)$,
\begin{displaymath}
  \begin{array}{llll}
    \norm{\mu}_{h,\partial T}\!\!&\!\!:=\!\!&\!\! h_T^{\frac{1}{2}}\norm{\mu}_{\partial T},& \\
    \vertiii{\mu}_{h,\partial T}\!\! &\!\!:=\!\!&\!\! h_T^{-\frac{1}{2}}\norm{\mu-m_T(\mu)}_{\partial T} & \text{ with }
    \quad m_T(\mu):=\frac{1}{d+1}\sum\limits_{F\in\mathcal F_T}\frac{1}{|F|}\int_F\mu,\\
    \vertiii{\mu}_h \!\!&\!\!:=\!\!&\!\! (\sum_{T\in\mathcal T_h}\vertiii{\mu}^2_{h,\partial T})^{\frac{1}{2}},&
  \end{array}
\end{displaymath}
where $\mathcal F_T:=\{F:F\subset\partial T\text{ is a face of $T$}\}$ and
$|F|$ denotes the $(d-1)$-dimensional Hausdorff measure of $F$.
\par
In the context, we use $x \lesssim y $ to denote $x \leqslant cy$, where $c$ is a positive constant
independent of $h$ which may be different at its each occurrence. The notation $x \sim y$ abbreviates
$x \lesssim y\lesssim x$.
For the bilinear form $d_h(\cdot,\cdot)$ in the system \eqref{eq:sys}, we shall make the following
abstract assumption.
\begin{assum}\label{ass:d_h}
  For any $\mu_h\in \mathbb M_{h,k}$, it holds
  \begin{equation}\label{eq:assum d_h}
    d_h(\mu_h,\mu_h) \sim \vertiii{\mu_h}^2_h.
  \end{equation}
\end{assum}
\begin{rem}
  This assumption is valid for many nonstandard finite element methods, as will be shown in Section
  \ref{sec_appl}. We note that the Schwarz preconditioner constructed in \cite{schwarz_pre} can also
  be extended to the system \eqref{eq:sys} under {\bf Assumption \ref{ass:d_h}}.
\end{rem}
Based on {\bf Assumption \ref{ass:d_h}}, we are ready to present an estimate that describes the
conditioning of the system \eqref{eq:sys}.
\begin{thm}\label{thm_cond}
  Suppose  $\mathcal T_h$ to be quasi-uniform.  Under {\bf Assumption \ref{ass:d_h}}, it holds
  \begin{equation}
    \norm{\mu_h}^2_h \lesssim d_h(\mu_h,\mu_h) \lesssim h^{-2}\norm{\mu_h}^2_h,~\forall\mu_h\in \mathbb M_{h,k}.
  \end{equation}
\end{thm}
\begin{proof}
  By Lemma 3.1 of \cite{Li;Xie;2014;WG}, we have
  \begin{equation}
    \norm{\mu_h}^2_h\lesssim\vertiii{\mu_h}_h^2.
  \end{equation}
  Then the desired conclusion follows from    {\bf Assumption \ref{ass:d_h}} and the fact $\vertiii{\mu_h}^2_h\lesssim h^{-2}\norm{\mu_h}^2_h$.
\end{proof}
\begin{rem}
  In Theorem 2.3 of \cite{schwarz_pre}, a similar result was derived in the two-dimensional case. But the proof there
  could not be extended to three-dimensional case directly.
\end{rem}

We define the operator $D_h:\mathbb M_{h,k}\to \mathbb M_{h,k}$ by
\begin{equation}\label{def_D_h}
  \inpro{D_h\lam_h,\mu_h} := d_h(\lam_h,\mu_h)~~~\text{ for all $\lam_h,\mu_h\in \mathbb M_{h,k}$}.
\end{equation}
Obviously, $D_h$ is an SPD operator and, from Theorem \ref{thm_cond}, it follows the condition
number estimate
\begin{equation}\label{cond-Dh}
  \kappa(D_h)\lesssim h^{-2},
\end{equation}
where $\kappa(D_h) := \frac{\lam_{\text{\text{max}}}(D_h)}{\lam_{\text{min}}(D_h)}$ and
$\lam_{\text{max}}(D_h),\lam_{\text{min}}(D_h)$ denote the maximum and minimum eigenvalues of $D_h$, respectively.
In fact, with a slight modification of the proof of Theorem 2.1 of \cite{Li;Xie;2014;WG}, we can show that
$\kappa (D_h) \sim h^{-2}$ holds under the condition that $h$ is sufficiently small.

\section{Auxiliary space preconditioning}\label{sec:aux_pre}
In this section, we shall follow the basic idea of \cite{Xu;1996} to introduce a general auxiliary
space preconditioner for $D_h$. It should be stressed that we only require the triangulation $\mathcal T_h$
to be conforming and shape regular.

Let $V$ be a finite dimensional Hilbert space endowed with
inner product $(\cdot,\cdot)$ and its induced norm $\norm{\cdot}$. Let $S:V\to V$ be SPD with respect to
$(\cdot,\cdot)$. We use $(\cdot,\cdot)_S$ to denote the inner product $(S\cdot,\cdot)$, and use $\norm{\cdot}_S$
to denote the norm induced by $(\cdot,\cdot)_S$.


We choose the $H^1$-conforming piecewise linear element space as the  so-called auxiliary space $V^c_h$, namely
\begin{equation}
  V^c_h:=\{v_h\in H^1_0(\Omega):v_h|_T\in P_1(T)~\text{ for all $T\in\mathcal T_h$}\}.
\end{equation}
Then we introduce two different prolongation operators that map $V^c_h$ into $\mathbb M_{h,k}$ as follows:
\begin{itemize}
  \item
    $\Pi_h^1:V^c_h\to \mathbb M_{h,k}$ is defined by
    \begin{equation}\label{eq:def Pi_h^1}
      \Pi_h^1 v_h |_F := \frac{1}{|F|}\int_Fv_h~~~\text{ for all $F\in\mathcal F_h$ and $v_h\in V^c_h$}.
    \end{equation}
  \item
    $\Pi_h^2:V^c_h\to\mathbb M_{h,k}$ is defined by
    \begin{equation}\label{eq:def Pi_h^2}
      \int_F\Pi_h^2v_hq := \int_Fv_hq~~~\text{ for all $F\in\mathcal F_h, v_h\in V^c_h$ and $q\in P_k(F)$}.
    \end{equation}
\end{itemize}
Obviously, $\Pi_h^1$ coincides with $\Pi_h^2$ in the case that $k=0$. For the sake of convenience, in the rest of this
paper, unless otherwise specified, we shall use $\Pi_h$ to denote both $\Pi_h^1$ and $\Pi_h^2$ at the same time.
Define the adjoint operator, $\Pi_h^t:\mathbb M_{h,k}\to V^c_h$, of $\Pi_h$, by
\begin{equation}\label{eq:def Pi_h^t}
  (\Pi_h^t\mu_h,v_h) := \inpro{\mu_h,\Pi_h v_h}~~~\text{ for all $\mu_h\in\mathbb M_{h,k}$ and $v_h\in V^c_h$.}
\end{equation}

Before defining the auxiliary space preconditioner, we introduce two linear operators, $S_h$ and $\we{B_h}$, in the following two assumptions.
\begin{assum}\label{ass:S_h}
  Let   $ S_h:M_{h,k}\to M_{h,k}$ be SPD with respect to $\inpro{\cdot,\cdot}$ and satisfy the following estimates: for all $\mu_h\in M_{h,k}$,
  \begin{eqnarray}
    &\inpro{S_h\mu_h,\mu_h}\lesssim  \inpro{D_h^{-1}\mu_h,\mu_h},  \label{eq:ass S_h 1}&\\ 
    &\norm{\mu_h}^2_{S_h^{-1}}\lesssim\sum_{T\in T_h}h_T^{-2}\norm{\mu_h}^2_{h,\partial T}.
    & \label{eq:ass S_h 2}
  \end{eqnarray}
\end{assum}

\begin{assum}\label{ass:we B_h}
  Let $\we{B_h}:V^c_h\to V^c_h$ be SPD with respect to $(\cdot,\cdot)$ and satisfy the
  estimate
  $$
  (\we{B_h}^{-1}v_h,v_h)\sim\verti{v_h}^2_{1,\Omega}~~~\text{ for all $v_h\in V^c_h$}.
  $$
\end{assum}

Then we define the general auxiliary space preconditioner $ B^G_h:M_{h,k}\to M_{h,k}$ by
\begin{equation}\label{eq:def B^G_h}
  B^G_h:= S_h+\Pi_h\we{ B_h}\Pi_h^t.
\end{equation}
\begin{rem} We note that the Jacobi iteration and the symmetric Gauss-Seidel iteration
  satisfy {\bf Assumption \ref{ass:S_h}} if $\mathcal T_h$ is conforming and shape regular,
  while the Richardson iteration does if $\mathcal T_h$ is quasi-uniform.
\end{rem}

\begin{rem}
  The preconditioner $B^G_h$ was also analyzed recently in \cite{Chen-Wang-Wang-Ye2014}
  for two types of WG methods, where $\mathcal T_h$ is assumed to be quasi-uniform.
  In our analysis below we only require $\mathcal T_h$ to be conforming and shape regular.
  We refer to Theorem 2.1 of \cite{Xu;1996} for a more general result  for auxiliary space
  preconditioning under quasi-uniform meshes.
\end{rem}
For the auxiliary space preconditioner $B_h^G$, we have the following main result.
\begin{thm}\label{thm:aux_pre}
  Under {\bf Assumptions \ref{ass:d_h}, \ref{ass:S_h}} and {\bf \ref{ass:we B_h}}, it holds
  \begin{equation}\label{eq:cond_psc}
    \kappa( B^G_hD_h)\lesssim 1,
  \end{equation}
  where $\kappa(B^G_hD_h):=\frac{\lam_{\text{max}}(B^G_hD_h)}{\lam_{\text{min}}(B^G_hD_h)}$, and
  $\lam_{\text{max}}(B^G_hD_h)$ and $\lam_{\text{min}}(B^G_hD_h)$ denote the maximum and minimum
  eigenvalues of $B^G_hD_h$, respectively.
\end{thm}
\begin{rem}
  Since we only require $\mathcal T_h$ to be conforming and shape regular, Theorem \ref{thm:aux_pre}
  is not a trivial application of Theorem 2.1 in \cite{Xu;1996}.
\end{rem}

Before proving Theorem \ref{thm:aux_pre}, we introduce a key ingredient operator $P_h:M_{h,k}\to V^c_h$
as follows: For each node $\bm a$ of $\mathcal T_h$,
\begin{equation}
  P_h\mu_h (\bm a) :=
  \left\{
    \begin{array}{ll}
      \frac{\sum_{T\in\omega_{\bm a}}m_T(\mu_h)}{\sum_{T\in\omega_{\bm a}}1}
      & \text{ if $\bm a$ is an interior node,}\\
      0 & \text{ if $\bm a\in\partial\Omega$,}
    \end{array}
  \right.
\end{equation}
where $\omega_{\bm a}$ denotes the set of simplexes that share the vertex $\bm a$.
\begin{lem}\label{lem:P_h}
  For any $\mu_h\in \mathbb M_{h,k}$, it holds
  \begin{equation}\label{eq_p_h}
    |P_h\mu_h|_{1,\Omega} \lesssim \vertiii{\mu_h}_h,
  \end{equation}
  \begin{equation}\label{311}
    \sum_{T\in\mathcal T_h}h_T^{-2}\norm{(I-\Pi_hP_h)\lam_h}^2_{h,\partial T}\lesssim \vertiii{\mu_h}^2_h.
  \end{equation}
\end{lem}
\begin{proof}
  For any $T\in\mathcal T_h$, we use $\mathcal N(T)$,  $\omega_T$ to denote the set of all vertexes of $T$ and
  the set $\{T'\in\mathcal T_h:T'\in\omega_{\bm a}\text{ for some $\bm a\in\mathcal N(T)$}\}$, respectively.
  For $\bm a\in\mathcal N(T)$, if $\bm a\in\Omega$, then we have
  \begin{displaymath}
    \begin{split}
      &h_T^{d-2}|m_T(\mu_h)-(P_h\mu_h)(\bm a)|^2\\
      \lesssim& h_T^{d-2}\sum_{\substack{T_1,T_2\in\omega_{\bm a}\\T_1, T_2 \text{ share a same face}}}
      |m_{T_1}(\mu_h)-m_{T_2}(\mu_h)|^2\\
      \lesssim& h_T^{-1}\sum_{\substack{T_1,T_2\in\omega_{\bm a}\\T_1, T_2 \text{ share a same face}}}
      \norm{m_{T_1}(\mu_h)-m_{T_2}(\mu_h)}^2_{\partial T_1\cap\partial T_2}\\
      \lesssim& h_T^{-1}\sum_{\substack{T_1,T_2\in\omega_{\bm a}\\T_1, T_2 \text{ share a same face}}}
      \bigg(\norm{\mu_h-m_{T_1}(\mu_h)}^2_{\partial T_1\cap\partial T_2}+\norm{\mu_h-m_{T_2}(\mu_h)}^2_{\partial T_1\cap\partial T_2}\bigg)\\
      \lesssim& h_T^{-1}\sum_{T'\in\mathcal\omega_{\bm a}}\norm{\mu_h-m_{T'}(\mu_h)}^2_{\partial T'}\\
      \lesssim& \sum_{T'\in\mathcal\omega_{\bm a}}\vertiii{\mu_h}^2_{h,\partial T'}.
    \end{split}
  \end{displaymath}
  If $\bm a\in\partial\Omega$, suppose that $F\subset\partial\Omega$ is a face of $T$ such that
  $\bm a\in\partial F$. Since $\mu_h|_{\partial\Omega}=0$, we have
  \begin{displaymath}
    \begin{split}
      h_T^{d-2}|m_T(\mu_h)-(P_h\mu_h)(\bm a)|^2
      =h_T^{d-2}|m_T(\mu_h)|^2
      \sim&h_T^{-1}\norm{m_T(\mu_h)}^2_F\\
      \lesssim&h_T^{-1}\norm{\mu_h-m_T(\mu_h)}^2_F\\
      \lesssim&\vertiii{\mu_h}^2_{h,\partial T}.
    \end{split}
  \end{displaymath}
  In light of the above two estimates, we immediately get
  \begin{equation}\label{eq:111}
    h_T^{d-2}\sum_{\bm a\in\mathcal N(T)}|m_T(\mu_h)-(P_h\mu_h)(\bm a)|^2
    \lesssim \sum_{T'\in\omega_T}\vertiii{\mu_h}^2_{h,\partial T'}.
  \end{equation}
  Since $m_T(\mu_h)$ is a constant on $T$,  it follows
  \begin{displaymath}
    \begin{array}{rll}
      |P_h\mu_h|^2_{1,T} &= |m_T(\mu_h)-P_h\mu_h|^2_{1,T}&\\
      &\lesssim h_T^{-2}\norm{m_T(\mu_h)-P_h\mu_h}^2_T&\text{ (by inverse estimate) }\\
      &\lesssim h_T^{d-2}\sum_{\bm a\in\mathcal N(T)}|m_T(\mu_h)-(P_h\mu_h)(\bm a)|^2&\\
      &\lesssim \sum_{T'\in\mathcal\omega_T}\vertiii{\mu_h}^2_{h,\partial T'},&\text{ (by \eqref{eq:111})}
    \end{array}
  \end{displaymath}
  which implies
  $$
  |P_h\mu_h|^2_{1,\Omega} = \sum_{T\in\mathcal T_h}|P_h\mu_h|^2_{1,T} \lesssim \vertiii{\mu_h}_h^2,
  $$
  i.e., the estimate \eqref{eq_p_h} holds.
  \par

  We recall that $\mathcal F_T$ is the set of all faces of $T$.  For any $F\in\mathcal F_T$, we use $\mathcal N(F)$ to denote the set of all vertexes of $F$.
  Since
  \begin{displaymath}
    \begin{split}
      \norm{m_T(\mu_h)-\Pi_hP_h\mu_h}^2_{\partial T}
      &=\sum_{F\in\mathcal F_T}\norm{m_T(\mu_h)-\Pi_hP_h\mu_h}^2_F\\
      &\lesssim h_T^{d-1}\sum_{F\in\mathcal F_T}\sum_{\bm a\in\mathcal N(F)}|m_T(\mu_h)-P_h\mu_h(\bm a)|^2
      ~~~\text{(by  \eqref{eq:def Pi_h^1} and \eqref{eq:def Pi_h^2})}\\
      &\lesssim h_T^{d-1}\sum_{\bm a\in\mathcal N(T)}|m_T(\mu_h)-P_h\mu_h(\bm a)|^2\\
      &\lesssim h_T\sum_{T'\in\omega_T}\vertiii{\mu_h}^2_{h,T'},~~~\text{(by \eqref{eq:111})}
    \end{split}
  \end{displaymath}
  we get
  \begin{displaymath}
    \begin{split}
      \norm{\mu_h-\Pi_hP_h\mu_h}^2_{\partial T} &\lesssim h_T\vertiii{\mu_h}^2_{h,\partial T} +
      \norm{m_T(\mu_h)-\Pi_hP_h\mu_h}^2_{\partial T}\\
      &\lesssim\sum_{T'\in\omega_T}h_{T'}\vertiii{\mu_h}^2_{h,\partial T'}.
    \end{split}
  \end{displaymath}
  Therefore,
  $$
  \norm{(I-\Pi_hP_h)\mu_h}^2_h=\sum_{T\in\mathcal T_h}h_T\norm{(I-\Pi_hP_h)\mu_h}^2_{\partial T}
  \lesssim h^2\vertiii{\mu_h}_h^2,
  $$
  i.e.  \eqref{311} holds.
\end{proof}
Now, we are ready to prove Theorem \ref{thm:aux_pre}.\\
{\bf Proof of Theorem \ref{thm:aux_pre}}.
For any $T\in\mathcal T_h$,  standard scaling arguments yield
\begin{equation}\label{eq:500}
  \vertiii{\Pi_hv}_{h,\partial T}\sim\verti{v}_{1,T}~~\text{ for all $v\in P_1(T)$}.
\end{equation}
Define $\we{D_h}:= \Pi_h^tD_h\Pi_h$. Then, for any $v_h\in V^c_h$, we have
\begin{displaymath}
  \begin{array}{rll}
    (\we{D_h}v_h,v_h)
    &=\ninpro{\Pi_hv_h,\Pi_hv_h}{D_h}&\\
    &\sim\sum_{T\in T_h}\vertiii{\Pi_hv_h}^2_{h,\partial T}&\text{(by {\bf Assumption \ref{ass:d_h}})}\\
    &\sim\sum_{T\in T_h}\verti{v_h}^2_{1,T}&\text{(by \eqref{eq:500})}\\
    &\sim(\we{B_h}^{-1}v_h,v_h),&\text{(by {\bf Assumption \ref{ass:we B_h}})}
  \end{array}
\end{displaymath}
i.e.,
\begin{equation}\label{eq:equi append}
  (\we{D_h}v_h,v_h)\sim(\we{B_h}^{-1}v_h,v_h)~\text{ for all $v_h\in V^c_h$}.
\end{equation}

By the definition of $ B^G_h$, it holds, for any $\mu_h\in M_{h,k}$,
\begin{displaymath}
  \begin{split}
    \ninpro{ B^G_h D_h\mu_h,\mu_h}{D_h}
    &=\ninpro{S_h D_h\mu_h,\mu_h}{D_h} +
    (\we{B_h}\Pi_h^t D_h\mu_h,\Pi_h^t D_h\mu_h)\\
    &\lesssim\norm{\mu_h}^2_{D_h} +
    (\we{B_h}\Pi_h^tD_h\mu_h,\Pi_h^tD_h\mu_h)~~~~~~~\text{(by {\bf Assumption \ref{ass:S_h}})}\\
    &\lesssim\norm{\mu_h}^2_{D_h} +
    (\we{D_h}^{-1}\Pi_h^tD_h\mu_h,\Pi_h^tD_h\mu_h)
    ~~~~\text{(by \eqref{eq:equi append})}\\
    &\lesssim\norm{\mu_h}^2_{D_h} +
    \norm{\we{D_h}^{-1}\Pi_h^tD_h\mu_h}^2_{\we{D_h}},
  \end{split}
\end{displaymath}
which, together with
\begin{displaymath}
  \begin{split}
    \norm{\we{D_h}^{-1}\Pi_h^tD_h\mu_h}_{\we{D_h}}
    &=\sup_{v_h\in V^c_h}\frac{(\we{D_h}^{-1}\Pi_h^tD_h\mu_h,v_h)_{\we{D_h}}}{\norm{v_h}_{\we{D_h}}}\\
    &=\sup_{v_h\in V^c_h}\frac{(\mu_h,\Pi_hv_h)_{D_h}}{\norm{v_h}_{\we{D_h}}}\\
    &\leqslant\sup_{v_h\in V^c_h}\frac{\norm{\mu_h}_{D_h}\norm{\Pi_hv_h}_{D_h}}{\norm{v_h}_{\we{D_h}}}\\
    &=\norm{\mu_h}_{D_h},
  \end{split}
\end{displaymath}
yields
\begin{equation}
  \ninpro{B^G_hD_h\mu_h,\mu_h}{D_h}\lesssim\norm{\mu_h}^2_{D_h}~\text{ for all $\mu_h\in M_{h,k}$}.
\end{equation}
Thus it follows
\begin{equation}\label{eq:ubnd_psc append}
  \lam_{max}( B^G_hD_h) \lesssim 1.
\end{equation}

On the other hand, by Theorem 1 of \cite{Chen;2010}, we have, for any $\lam_h\in M_{h,k}$,
\begin{displaymath}
  \begin{split}
    &\quad\inpro{(B^G_h)^{-1}\lam_h,\lam_h}\\
    &=\inf_{\mu_h+\Pi_hv_h=\lam_h}\inpro{S_h^{-1}\mu_h,\mu_h} +
    (\we{ B_h}^{-1}v_h,v_h)\\
    &\leqslant\norm{(I-\Pi_hP_h)\lam_h}^2_{S_h^{-1}}+\norm{P_h\lam_h}^2_{\we{B_h}^{-1}}\\
    &\lesssim\sum_{T\in T_h}h_T^{-2}\norm{(I-\Pi_hP_h)\lam_h}^2_{h,\partial T} +
    \verti{P_h\lam_h}^2_{1,\Omega}~~~\text{(by {\bf Assumptions \ref{ass:S_h}-\ref{ass:we B_h}})}\\
    &\lesssim\norm{\lam_h}^2_{D_h},~~~~~\text{(by Lemma \ref{lem:P_h})}
  \end{split}
\end{displaymath}
which implies
\begin{equation}\label{eq:lbnd_psc append}
  \lam_{min}(B^G_hD_h)\gtrsim 1.
\end{equation}
As a result, the desired estimate \eqref{eq:cond_psc} follows immediately from \eqref{eq:ubnd_psc append} and \eqref{eq:lbnd_psc append}.
This finishes the proof.

\section{BPX preconditioner}
\subsection{Preconditioner construction}
Suppose we are given a coarse quasi-uniform triangulation $\mathcal T_0$. Then we obtain a nested
sequence of triangulations $\{\mathcal T_j:0\leqslant j\leqslant J\}$ through a successive refinement
process, i.e., $\mathcal T_j$ is the uniform refinement of $\mathcal T_{j-1}$ for $j = 1,2,\ldots,J$.
We use $h_j$ to denote the mesh size of $\mathcal T_j$, i.e., the maximum diameter of the simplexes in
$\mathcal T_j$. For each triangulation $\mathcal T_j$, we define $V^c_j$ by
\begin{equation}
  V^c_j:=\{v\in H_0^1(\Omega): v|_T\in P_1(T)~\text{for all $T\in\mathcal T_j$}\},
\end{equation}
and let $\{\phi_{j,i}:i=1,2,\cdots, N_j\}$ be the standard nodal basis of $V^c_j$, where $N_j$ is the
dimension of $V^c_j$. We set $\{\eta_i:i=1,2,\ldots,M\}$ to be the standard nodal basis of $\mathbb M_{h,k}$.
Set $h=h_J$, $\mathcal T_h=\mathcal T_J$ and $V^c_h=V^c_J$.

With the operators $\Pi_h$ (defined by \eqref{eq:def Pi_h^1} or \eqref{eq:def Pi_h^2}), $\Pi_h^t$ (defined by
\eqref{eq:def Pi_h^t}),  the nodal basis, $\{\phi_{j,i}:i=0,1,\ldots, N_j\}$, of $V^c_j$,
and the nodal basis, $\{\eta_i:i=1,2,\ldots,M\}$, of  $\mathbb M_{h,k}$, we define the BPX preconditioner
(in operator form) for the operator $D_h$ given in \eqref{def_D_h} as follows:
\begin{equation}\label{def_B_h}
  B_h\mu_h= h^{2-d}\sum_{i=1}^M\inpro{\mu_h,\eta_i}\eta_i+\sum_{(j,i)\in\Lambda}h_j^{2-d}
  (\Pi_h^t\mu_h,\phi_{j,i})\Pi_h\phi_{j,i}~\text{ for all $\mu_h\in\mathbb M_{h,k}$},
\end{equation}
where $\Lambda:=\{(j,i):0\leqslant j \leqslant J, 1\leqslant i \leqslant N_j\}$. It is trivial to verify
that $B_h$ is SPD with respect to $\inpro{\cdot,\cdot}$.
\begin{rem}
  We shall prove in the next subsection that both $\Pi_h^1$ and $\Pi_h^2$ lead to optimal preconditioners
  in the case $k\geqslant 1$, although  numerical results in Section \ref{sec:numer} show that $\Pi_h^2$
  is much more efficient. We note that $\Pi_h^2$ was also used in \cite{G_2009}, \cite{multigrid_hdg} and \cite{Li;Xie;2015;HDG}
  to construct multilevel methods for HDG methods.
\end{rem}
\subsection{Conditioning of \texorpdfstring{$B_hD_h$}{Lg}}
In this subsection, we shall use the framework of auxiliary space preconditioning   introduced in Section \ref{sec:aux_pre}
to analyze the BPX conditioning of $B_hD_h$.
For the sake of convenience, in this subsection we assume
\begin{displaymath}
  \begin{array}{rl}
    \mu_i\in\text{span}\{\eta_i\} & \text{for all } i=1,\cdots, M.
  \end{array}
\end{displaymath}

We define   two  SPD operators $S_h:M_{h,k}\to M_{h,k}$ and $\we{B_h}:V^c_h\to V^c_h$, respectively as follows:
\begin{equation}\label{eq:def S_h}
  S_h\mu_h:=h^{2-d}\sum_{i=1}^M\inpro{\mu_h,\eta_i}\eta_i~~\text{ for all $\mu_h\in M_{h,k}$},
\end{equation}
\begin{equation}\label{eq:def B_h}
  \we{B_h}v_h = \sum_{(j,i)\in\Lambda}h_j^{2-d}(v_h,\phi_{j,i})\phi_{j,i}~~\text{ for all $v_h\in V^c_h$}.
\end{equation}
Apparently we have
\begin{equation}
  B_h = S_h+\Pi_h\we{B_h}\Pi_h^t.
\end{equation}
Thus, according to Theorem \ref{thm:aux_pre}, to show   $\kappa(B_hD_h)\lesssim 1$ it suffices
to prove that $S_h$ and $\we{B_h}$ satisfy {\bf Assumption \ref{ass:S_h}} and {\bf Assumption \ref{ass:we B_h}},
respectively.
\begin{lem}
  The operator $S_h$ defined by \eqref{eq:def S_h} satisfies {\bf Assumption \ref{ass:S_h}}.
\end{lem}
\begin{proof}
  For any $\mu_h\in M_{h,k}$, by using  the same technique as  in the proof of Lemma 2.4 in \cite{xu_2002}, it is easy
  to verify
  \begin{equation}\label{eq:S_h^{-1}}
    \inpro{S_h^{-1}\mu_h,\mu_h}
    =\inf_{\sum_i\mu_i=\mu_h}\sum_i\frac{h^{d-2}}{\norm{\eta_i}^2_h}\norm{\mu_i}^2_h.
  \end{equation}
  By a  standard scaling argument, it holds $\norm{\eta_i}_h\sim h^{\frac{d}{2}}$. Then  from \eqref{eq:S_h^{-1}}
  it follows
  \begin{equation}
    \inpro{S_h^{-1}\mu_h,\mu_h}\sim h^{-2}\inf_{\sum_i\mu_i=\mu_h}\sum_i\norm{\mu_i}^2_h.
  \end{equation}
  Further more, since $\mathcal T_h$ is shape regular, a standard scaling argument
  also yields   $\sum_i\norm{\mu_i}^2_h\sim\norm{\sum_i\mu_i}^2_h$. Thus it
  holds
  \begin{equation}\label{eq:S_h^{-1} 2}
    \inpro{S_h^{-1}\mu_h,\mu_h}\sim h^{-2}\norm{\mu_h}^2_h,
  \end{equation}
  and  the desired estimate  \eqref{eq:ass S_h 2} follows from \eqref{eq:S_h^{-1} 2} immediately by the   quasi-uniform assumption of $\mathcal T_h$.

  The thing left is  to prove \eqref{eq:ass S_h 1}. In fact, from
  \begin{displaymath}
    \begin{split}
      \inpro{D_h\mu_h,\mu_h}
      &\sim\vertiii{\mu_h}^2_h~~~~\text{ (by {\bf Assumption \ref{ass:d_h}})}\\
      &\lesssim h^{-2}\norm{\mu_h}^2_h\\
      &\lesssim  \inpro{S_h^{-1}\mu_h,\mu_h}, ~~~~\text{(by \eqref{eq:S_h^{-1} 2})}
    \end{split}
  \end{displaymath}
  which implies immediately $    \inpro{S_h\mu_h,\mu_h} \lesssim \inpro{D_h^{-1}\mu_h,\mu_h}.$ Hence, \eqref{eq:ass S_h 1} follows immediately.
\end{proof}

The following lemma    follows from Theorems 3.1-3.2 of \cite{xu_2009}.
\begin{lem}\label{lem44}
  The operator  $\we{B_h}$ defined by \eqref{eq:def B_h} satisfies {\bf Assumption \ref{ass:we B_h}}.
\end{lem}

Finally, thanks to Theorem \ref{thm:aux_pre}, we obtain immediately the following main conclusion for the BPX preconditioning.
\begin{thm}\label{thm_cond_BD}
  Under {\bf Assumption \ref{ass:d_h}}, it holds
  \begin{equation}
    \kappa(B_hD_h)\lesssim 1,
  \end{equation}
  where $D_h$ and $B_h$ are defined by \eqref{def_D_h} and \eqref{def_B_h},
  respectively.
\end{thm}
\subsection{Implementation}
We recall that $\{\eta_i:1\leqslant i\leqslant M\}$ is the standard nodal basis of $\mathbb M_{h,k}$
and $\{\phi_{j,i}:i=0,1,\ldots, N_j\}$ is the standard nodal basis of $V^c_j$ for $j=0,1,\cdots,J$.
For each $\mu_h\in \mathbb M_{h,k}$, we use $\widetilde\mu_h\in\mathbb R^M$ to denote the vector of
coefficients of $\mu_h$ with respect to the basis $\{\eta_1,\eta_2,\ldots,\eta_M\}$. Let
$\mathcal D_h\in\mathbb R^{M\times M}$ be the stiffness matrix with respective to the operator $D_h$
defined in \eqref{def_D_h} with
\begin{displaymath}
  \widetilde\lam_h^{T}\mathcal D_h\widetilde\mu_h:=
  \inpro{D_h\mu_h,\eta_h}~~~\text{ for all $\lam_h,\mu_h\in \mathbb M_{h,k}$}.
\end{displaymath}
Then it follows from Theorem \ref{thm_cond}, or the estimate \eqref{cond-Dh},  that
\begin{displaymath}
  \kappa (\mathcal D_h) \lesssim h^{-2}.
\end{displaymath}
By the definition, \eqref{eq:def Pi_h^1}, of $\Pi_h$, there exists a matrix $\mathcal I_j\in\mathcal R^{M\times N_j}$
for $j=0,1,\cdots,J$, such that
\begin{equation}\label{I_j}
  \Pi_h(\phi_{j,1},\phi_{j,2},\ldots,\phi_{j,N_j}) = (\eta_1,\eta_2,\ldots,\eta_M)\mathcal I_j.
\end{equation}
We set $\mathcal I_h\in\mathbb R^{M\times M}$ to be the identity matrix.
From the definition, \eqref{def_B_h}, of $B_h$, it follows,  for any $\mu_h\in \mathbb M_{h,k}$,
\begin{displaymath}
  \begin{split}
    B_hD_h\mu_h &= h^{2-d}\sum_{i=1}^M\inpro{D_h\mu_h,\eta_i}\eta_i+\sum_{(j,i)\in\Lambda}h_j^{2-d}
    (\Pi_h^tD_h\mu_h,\phi_{j,i})\Pi_h\phi_{j,i}\\
    &= h^{2-d}\sum_{i=1}^M\inpro{D_h\mu_h,\eta_i}\eta_i + \sum_{(j,i)\in\Lambda}h_j^{2-d}
    \inpro{D_h\mu_h,\Pi_h\phi_{j,i}}\Pi_h\phi_{j,i}.
  \end{split}
\end{displaymath}
Thus, in view of \eqref{I_j}, we have
\begin{equation}
  \widetilde{B_hD_h\mu_h} = \mathcal B_h\mathcal D_h\widetilde\mu_h,~\forall\mu_h\in \mathbb M_{h,k},
\end{equation}
where $\mathcal B_h$, a preconditioner for $\mathcal D_h$, is given by
\begin{equation}\label{def_mat_B_h}
  \mathcal B_h = h^{2-d}\mathcal I_h + \sum_{k=0}^Jh_j^{2-d}\mathcal I_j\mathcal I_j^t.
\end{equation}
From Theorem \ref{thm_cond_BD} it follows
\begin{equation}\label{cond-BD}
  \kappa(\mathcal B_h\mathcal D_h) \lesssim 1.
\end{equation}
This means that the matrix $\mathcal B_h$ is an optimal preconditioner for the stiffness matrix
$\mathcal D_h$.
\section{Applications}\label{sec_appl}
We begin by introducing some notation. For any $T\in\mathcal T_h$, let $V(T)$ and $\bm W(T)$ be two local
finite dimensional spaces. Define
\begin{displaymath}
  \begin{split}
    V_h &:= \{v\in L^2(\Omega): v_h|_T \in V(T)\text{ for all $T\in\mathcal T_h$}\},\\
    \bm W_h &:= \{\bm\tau\in [L^2(\Omega)]^d: \bm\tau_h|_T \in \bm W(T)\text{ for all $T\in\mathcal T_h$}\}.
  \end{split}
\end{displaymath}
Then we introduce another local space as follows:
$$M(\partial T):=\left\{\mu \in L^2(\partial T): \mu |_F \in P_k(F) \text{ for all $F\in\mathcal F_T$}\right\}.$$
We recall
\begin{equation}
  \mathbb M_{h,k}:=\{\mu_h\in L^2( \mathcal F_h):\mu_h|_F\in P_k(F)~\text{ for all $F\in\mathcal F_h$
    and $\mu_h|_{\partial\Omega}=0$}\}.
\end{equation}
For the sake of clarity, in what follows we assume $g=0$ for the model problem \eqref{eq:model}.

\subsection{Hybridized discontinuous Galerkin method}\label{ssec:HDG}
The general framework of HDG method for the problem
\eqref{eq:model} reads as follows (\cite{unified_hdg}): Seek
$(u_h,\lambda_h,\bm\sigma_h) \in V_h \times \mathbb M_{h,k} \times \bm W_h$ such that
\begin{subequations}\label{discretization_hdg}
  \begin{eqnarray}
    (\bm C\bmsigma_h,\bm\tau_h)+(u_h,\text{div}_h\bm\tau_h)-\bndint{\lambda_h,\bm\tau_h\cdot\bm n} &=& 0,\label{discrete1} \\
    -(v_h,\text{div}_h\bm\sigma_h) + \bndint{\alpha_T(P_T^{\partial}u_h - \lambda_h), v_h} &=& (f, v_h), \label{discrete2}\\
    \bndint{\bm\sigma_h\cdot\bm n - \alpha_T(P_T^{\partial}u_h-\lambda_h),\mu_h} &=& 0\label{discrete3}
  \end{eqnarray}
\end{subequations}
hold for all $(v_h,\mu_h,\bm\tau_h)\in V_h\times \mathbb M_{h,k} \times \bm W_h$, where $\bm C = \bm A^{-1}$, $\text{div}_h$
is the broken $\text{div}$ operator with respective to the triangulation $\mathcal T_h$, $\bm n$ denotes the unit outward normal
of $T$, $P_T^{\partial}:H^1(T)\to M(\partial T)$ denotes the standard $L^2$-orthogonal projection operator,
and $\alpha_T$ denotes a nonnegative penalty function defined on $\partial T$.

For any $T\in\mathcal T_h$, we introduce two local problems as follows.
\par
{\bf Local problem 1}: For any given $\lam\in L^2(\partial T)$, seek $(u_{\lam},\bmsigma_{\lam})\in V(T)\times \bm W(T)$
such that
\begin{subequations}\label{eq:local_1 hdg}
  \begin{eqnarray}
    (\bm C\bmsigma_{\lam},\bm\tau)_T+(u_{\lam},\text{div}\bm\tau)_T\!\! &\!\!=\!\!&\!\!
    \bint{\lam,\bm\tau\cdot\bm n}{\partial T},\label{eq:local_1 hdg_a}\\
    -(v,\text{div}\bmsigma_{\lam})_T + \bint{\alpha_TP_T^{\partial}u_{\lam},v}{\partial T}\!\!&\!\!=\!\!&\!\!
    \bint{\alpha_T\lam, v}{\partial T},\label{eq:local_1 hdg_b}
  \end{eqnarray}
\end{subequations}
hold for all $(v,\bm\tau)\in V(T)\times \bm W(T)$.

{\bf Local problem 2}: For any given $f\in L^2(T)$, seek $(u_f,\bmsigma_f)\in V(T)\times \bm W(T)$ such that
\begin{subequations}\label{eq:local2 hdg}
  \begin{eqnarray}
    (\bm C\bmsigma_f,\bm\tau)_T+(u_f,\text{div}\bm\tau)_T\!\! &\!\!=\!\!&\!\! 0,\label{eq:local2 hdg_a}\\
    -(v,\text{div}\bmsigma_f)_T + \bint{\alpha_TP_T^{\partial}u_f,v}{\partial T}\!\!&\!\!=\!\!&\!\! (f,v)_T,\label{eq:local2 hdg_b}
  \end{eqnarray}
\end{subequations}
hold for all $(v,\bm\tau)\in V(T)\times \bm W(T)$.\par
\begin{thm} \cite{unified_hdg}\label{global-local}
  Suppose $(u_h,\lam_h,\bmsigma_h)\in V_h\times \mathbb M_{h,k} \times \bm W_h$ to be the solution
  to the system \eqref{discretization_hdg}, and suppose, for any $T\in \mathcal T_h$, $(u_{\lam_h},
  \bmsigma_{\lam_h})|_T\in V(T)\times\bm W(T)$ and $(u_{f},\bmsigma_{f})|_T\in V(T)\times \bm W(T)$
  to be the solutions to the local problems
  \eqref{eq:local_1 hdg}  (by replacing $\lam$ with $\lam_h$) and \eqref{eq:local2 hdg}, respectively. Then it holds
  \begin{eqnarray}
    \bmsigma_h\!\! &\!=\!&\!\! \bmsigma_{\lam_h} + \bmsigma_f,\\
    u_h\!\!&\!=\!&\!\! u_{\lam_h} + u_f,
  \end{eqnarray}
  and $\lam_h\in \mathbb M_{h,k}$ is the solution to the system \eqref{eq:sys}, i.e.
  \begin{equation*}
    d_h(\lam_h,\mu_h) = b_h(\mu_h)~~~\text{ for all $\mu_h \in \mathbb M_{h,k}$},
  \end{equation*}
  where
  \begin{eqnarray}
    d_h(\lam_h,\mu_h)\!&\!\!:=\!\!&\!(\bm C\bmsigma_{\lam_h},\bmsigma_{\mu_h}) + \bndint{\alpha_T(P_T^{\partial}u_{\lam_h}-\lam_h),
      P_T^{\partial}u_{\mu_h}-\mu_h},\label{def_d_h_hdg}\\
    b_h(\mu_h)\!&\!\!:=\!\!&\!(f,u_{\mu_h}),
  \end{eqnarray}
  and, for any $T\in\mathcal T_h$, $(u_{\mu_h},\bmsigma_{\mu_h})|_T\in V(T)\times\bm W(T)$ denotes the solution
  to the local problem \eqref{eq:local_1 hdg} by replacing $\lam$ with $\mu_h$.
\end{thm}
\par

We list four types of HDG methods as follows. 
\begin{description}
  \item[Type 1. ]  $V(T) = P_k(T)$, $\bm W(T) = [P_k(T)]^d + P_k(T)\bm x$ and $\alpha_T = 0$.
  \item[Type 2. ] $V(T) = P_{k-1}(T)~(k\geqslant 1)$, $\bm W(T) = [P_k(T)]^d$ and $\alpha_T = 0$.
  \item[Type 3. ] $V(T) = P_k(T)$, $\bm W(T) = [P_k(T)]^d$ and $\alpha_T = O(1)$.
  \item[Type 4. ] $V(T) = P_{k+1}(T)$, $\bm W(T) = [P_k(T)]^d$ and $\alpha_T = O(h_T^{-1})$.
\end{description}

{\bf Type 1} HDG method turns out to be the well-known hybridized RT mixed element method (\cite{ArnoldBrezzi1985}), and
{\bf Type 2} HDG method turns out to be the well-known hybridized BDM mixed element method (\cite{BrezziDouglasMarini1985}).
For both {\bf Types 1-2}  HDG methods, it was shown in \cite{schwarz_pre} that {\bf Assumption \ref{ass:d_h}} holds.

{\bf Type 3} HDG method was proposed in \cite{unified_hdg} and analyzed in \cite{projection_based_hdg}.  In
\cite{multigrid_hdg} it was shown  that Assumption \ref{ass:d_h}  holds for this method.

{\bf Type 4} HDG method was proposed and analyzed in \cite{Li-Xie2014}. The proof of {\bf Assumption \ref{ass:d_h}} can also be found there.  For completeness we sketch the proof here.

\begin{lem}\label{lem51}
  Let $T\in\mathcal T_h$. For any given $\lam\in M(\partial T)$, it holds
  \begin{equation}\label{eq:367}
    (\bm C^{-1}\bmsigma_{\lam},\bmsigma_{\lam})_T+\bint{\alpha_T(P_T^{\partial}u_{\lam}-\lam),P_T^{\partial}u_{\lam}-\lam}{\partial T}
    \sim\vertiii{\lam}^2_{h,\partial T}.
  \end{equation}
\end{lem}
\begin{proof}
  We first show
  \begin{equation}\label{eq_cond_low}
    h_T^{-\frac{1}{2}}\norm{\lam-\bar\lam}_{\partial T}
    \lesssim \norm{\bmsiglam}_T+\norm{\alpha_T^{\frac{1}{2}}(P_T^{\partial}u_{\lam}-\lam)}_{\partial T},
  \end{equation}
  where $\bar\lam=\frac{1}{|\partial T|}\int_{\partial T}\lam$. In fact,
  from \eqref{eq:local_1 hdg_a}  it follows
  \begin{equation}\label{eq:400}
    (\bmnabla u_{\lam},\bm\tau)_T = (\bm C\bmsiglam,\bm\tau)_T+\bint{P_T^{\partial}u_{\lam}-\lam,\bm\tau\cdot\bm n}{\partial T}
    ~~\text{ for all $\bm\tau\in\bm W(T)$}.
  \end{equation}
  Taking $\bm\tau=\bmnabla u_{\lam}$ in \eqref{eq:400}, we  immediately  get
  \begin{equation}\label{eq:401}
    \verti{u_{\lam}}_{1,T}\lesssim\norm{\bmsiglam}+h_T^{-\frac{1}{2}}\norm{P_T^{\partial}u_{\lam}-\lam}_{\partial T}.
  \end{equation}
  Define $\tilde{\lam}:=\lam-\bar\lam$ and $\bar{u}_{\lam}:=\frac{1}{|T|}\int_Tu_{\lam}$, then we have
  \begin{displaymath}
    \begin{split}
      \bint{\tilde{\lam},u_{\lam}}{\partial T} &= \bint{\tilde{\lam},u_{\lam}-\bar{u}_{\lam}}{\partial T} \\
      &\leqslant \norm{\tilde{\lam}}_{\partial T} \norm{u_{\lam}-\bar{u}_{\lam}}_{\partial T}\\
      &\lesssim h_T^{\frac{1}{2}}\norm{\tilde{\lam}}_{\partial T}\verti{u_{\lam}}_{1,T}\\
      &\lesssim h_T^{\frac{1}{2}}\norm{\tilde{\lam}}_{\partial T}\left(\norm{\bmsiglam}_T +
        \norm{\alpha_T^{\frac{1}{2}}(P_T^{\partial}u_{\lam}-\lam)}_{\partial T}\right).~~~~~~~~~~~\text{by \eqref{eq:401}}
    \end{split}
  \end{displaymath}
  This estimate, together with
  \begin{displaymath}
    \begin{split}
      \bint{\tilde{\lam},\lam-P_T^{\partial}u_{\lam}}{\partial T} &\leqslant \norm{\tilde{\lam}}_{\partial T} \norm{\lam-P_T^{\partial}u_{\lam}}_{\partial T}\\
      &\lesssim h_T^{\frac{1}{2}}\norm{\tilde{\lam}}_{\partial T}\norm{\alpha_T^{\frac{1}{2}}(\lam-P_T^{\partial}u_{\lam})}_{\partial T},
    \end{split}
  \end{displaymath}
  yields
  \begin{displaymath}
    \begin{split}
      \norm{\tilde{\lam}}^2_{\partial T} &= \bint{\tilde{\lam},\tilde{\lam}-P_T^{\partial}u_{\lam}}{\partial T} + \bint{\tilde{\lam},P_T^{\partial}u_{\lam}}{\partial T}\\
      &= \bint{\tilde{\lam},\lam-P_T^{\partial}u_{\lam}}{\partial T} + \bint{\tilde{\lam}, u_{\lam}}{\partial T}\\
      &\lesssim h_T^{\frac{1}{2}}\norm{\tilde{\lam}}_{\partial T}\left(\norm{\bmsiglam}_T +
        \norm{\alpha_T^{\frac{1}{2}}(P_T^{\partial}u_{\lam}-\lam)}_{\partial T}\right),
    \end{split}
  \end{displaymath}
  which implies \eqref{eq_cond_low} immediately.

  Second, we show
  \begin{equation}\label{eq_cond_upper}
    \norm{\bmsiglam}_T + \norm{\alpha_T^{\frac{1}{2}}(P_T^{\partial}u_{\lam}-\lam)}_{\partial T} \lesssim
    h_T^{-\frac{1}{2}}\norm{\lam-\bar\lam}_{\partial T}.
  \end{equation}
  In fact,  taking $\bm\tau = \bmsiglam$ in \eqref{eq:local_1 hdg_a}, $v = u_{\lam}-\bar\lam$ in \eqref{eq:local_1 hdg_b}, and adding the two  resultant
  equations, we obtain
  \begin{displaymath}
    \begin{split}
      &\norm{\bm C^{\frac{1}{2}}\bmsiglam}_T^2 + \norm{\alpha_T^{\frac{1}{2}}(P_T^{\partial}u_{\lam}-\bar\lam)}^2_{\partial T} \\
      =~& \bint{\lam-\bar\lam,\bmsiglam\cdot\bm n}{\partial T} + \bint{\alpha_T(\lam-\bar\lam),P_T^{\partial} u_{\lam}-\bar\lam}{\partial T}\\
      \leqslant~& \norm{\lam-\bar\lam}_{\partial T}\norm{\bmsiglam}_{\partial T} + \norm{\alpha_T^{\frac{1}{2}}(\lam-\bar\lam)}_{\partial T}
      \norm{\alpha_T^{\frac{1}{2}}(P_T^{\partial} u_{\lam}-\bar\lam)}_{\partial T}\\
      \lesssim~& h_T^{-\frac{1}{2}}\norm{\lam-\bar\lam}_{\partial T}\norm{\bmsiglam}_T+\norm{\alpha_T^{\frac{1}{2}}(\lam-\bar\lam)}_{\partial T}
      \norm{\alpha_T^{\frac{1}{2}}(P_T^{\partial} u_{\lam}-\bar\lam)}_{\partial T}\\
      \lesssim~& h_T^{-\frac{1}{2}}\norm{\lam-\bar\lam}_{\partial T}\left(\norm{\bmsiglam}_T+\norm{\alpha_T^{\frac{1}{2}}(P_T^{\partial} u_{\lam}-\bar\lam)}_{\partial T}\right),
    \end{split}
  \end{displaymath}
  which implies
  \begin{equation}\label{8829383}
    \norm{\bmsiglam}_T + \norm{\alpha_T^{\frac{1}{2}}(P_T^{\partial}u_{\lam}-\bar\lam)}_{\partial T}
    \lesssim h_T^{-\frac{1}{2}}\norm{\lam-\bar\lam}_{\partial T}.
  \end{equation}
  By noticing that the above estimate also indicates
  \begin{displaymath}
    \begin{split}
      \norm{\alpha_T^{\frac{1}{2}}(P_T^{\partial}u_{\lam}-\lam)}_{\partial T}&\leqslant\norm{\alpha_T^{\frac{1}{2}}(P_T^{\partial}
        u_{\lam}-\bar\lam)}_{\partial T} + \norm{\alpha_T^{\frac{1}{2}}(\lam-\bar\lam)}_{\partial T}\\
      &\lesssim h_T^{-\frac{1}{2}}\norm{\lam-\bar\lam}_{\partial T},
    \end{split}
  \end{displaymath}
  the estimate \eqref{eq_cond_upper} follows immediately. Then,  from \eqref{eq_cond_low} and \eqref{eq_cond_upper},  it follows
  \begin{equation}\label{eq:1101}
    (\bm C^{-1}\bmsigma_{\lam},\bmsigma_{\lam})_T+\bint{\alpha_T(P_T^{\partial}u_{\lam}-\lam),P_T^{\partial}u_{\lam}-\lam}{\partial T}
    \sim h_T^{-1}\norm{\lam-\bar\lam}^2_{\partial T}.
  \end{equation}
  A standard scaling argument shows
  \begin{equation}
    \vertiii{\lam}_{h,\partial T}\sim h_T^{-\frac{1}{2}}\norm{\lam-\bar\lam}_{\partial T},
  \end{equation}
  which, together with   \eqref{eq:1101}, indicates the desired estimate \eqref{eq:367}.
\end{proof}

Based on Lemma \ref{lem51}, it is trivial to derive the proposition below.
\begin{propo}
  For {\bf Type 4} HDG method, {\bf Assumption \ref{ass:d_h}} holds.
\end{propo}
\begin{rem}
  It has been shown in  \cite{COCKBURN_2004,COCKBURN_2005} that, when $\bm A$ is
  a piecewise constant matrix and $k\geq 1$, the bilinear form $ d_h(\cdot,\cdot)$
  arising from the hybridized RT mixed element method, i.e. {\bf Type 1} HDG method,
  coincides with  that arising from the hybridized BDM mixed element method, i.e. {\bf Type 2}
  HDG method. Then any preconditioner for {\bf Type 1} HDG method is also a preconditioner
  for {\bf Type 2} HDG method, and vice versa.
\end{rem}
\begin{rem}
  In \cite{multigrid_hdg}, a first analysis of multigrid method for {\bf Type 3} HDG method was
  presented. However, it was required there that that the model problem \eqref{eq:model}
  admits the regulartiy estimate $\norm{u}_{1+\alpha,\Omega}\leqslant C_{\alpha,\Omega}\norm{f}_{\alpha-1,\Omega}$
  with $\alpha\in(0.5,1]$ and $C_{\alpha,\Omega}$ being a positive constant that only depends on $\alpha$
  and $\Omega$. We note that our analysis in Section 4 for the BPX preconditioner does not require
  any regularity assumption. In \cite{Li;Xie;2015;HDG}, a more general framework for designing multilevel methods for HDG methods
  were presented and analyzed without any regularity assumption.
\end{rem}
\subsection{ Weak Galerkin method}\label{ssec:WG}
At first, we follow \cite{WangYe2013} to introduce the discrete weak gradients.
Let $T\in\mathcal T_h$. We define $\bmnabla_w^i: L^2(T)\to\bm W(T)$ by
\begin{equation}
  (\bmnabla_w^i v, \bm q)_T := -(v, \text{div}~\bm q)_T~~~\text{ for all $v\in L^2(T)$ and $\bm q\in\bm W(T)$ },
\end{equation}
and  define $\bmnabla_w^b: L^2(\partial T)\to\bm W(T)$ by
\begin{equation}
  (\bmnabla_w^b \mu, \bm q)_T := \langle \mu, \bm q \cdot \bm n \rangle_{\partial T}
  ~~~~\text{ for all $\mu\in L^2(\partial T)$ and $\bm q \in\bm W(T)$.}
\end{equation}
Then we define the discrete weak gradients $\bmnabla_w : L^2(T)\times L^2(\partial T) \to \bm W(T)$ with
\begin{equation}
  \bmnabla_w(v, \mu) := \bmnabla_w^i v + \bmnabla_w^b \mu~~~\text{ for all $(v, \mu) \in L^2(T)\times L^2(\partial T)$}.
\end{equation}

Hence, the WG discretization reads as follows: Seek $(u_h, \lambda_h) \in V_h \times \mathbb M_{h,k}$ such that
\begin{equation}\label{eq:disc_wg}
  (\bm A\bmnabla_w(u_h,\lam_h),\bmnabla_w(v_h,\mu_h))+\bndint{\alpha_T(P_T^{\partial}u_h-\lam_h),P_T^{\partial}v_h-\mu_h} = (f,v_h)
\end{equation}
holds for all $(v_h, \mu_h) \in V_h \times \mathbb M_{h,k}$,
where $\alpha_T$ denotes a nonnegative penalty function defined on $\partial T$.

We shall follow the same routine as in the previous subsection to show a new characterization of the WG method.
We introduce two local problems as follows.
\par
{\bf Local problem 1'}: For any given $f\in L^2(T)$, seek $u_f\in V(T)$ such that
\begin{equation}\label{eq:local_1 hdg_wg}
  (\bm A\bmnabla_w^iu_f,\bmnabla_w^iv)_T + \langle\alpha_TP_T^{\partial}u_f,P_T^{\partial}v\rangle_{\partial T}
  = (f, v)_T
\end{equation}
holds for all $v\in V(T)$.
\par
{\bf Local problem 2'}: For any given $\lam\in L^2(\partial T)$, seek $u_{\lam}\in V(T)$ such that
\begin{equation}\label{eq:local2 hdg_wg}
  (\bm A\bmnabla_w^iu_{\lam},\bmnabla_w^iv)_T + \langle\alpha_TP_T^{\partial}u_{\lam},P_T^{\partial}
  v\rangle_{\partial T}=-(\bm A\bmnabla_w^b\lam,\bmnabla_w^iv)_T + \langle\alpha_T\lam,P_T^{\partial}v
  \rangle_{\partial T}
\end{equation}
holds for all $v\in V(T)$.

Similar to Theorem \ref{global-local},  the following conclusion holds.
\begin{thm}
  Suppose $(u_h,\lam_h)\in V_h\times \mathbb M_{h,k}$ to be the solution to the
  system \eqref{eq:disc_wg}, and suppose, for any $T\in\mathcal T_h$, $u_f$
  and $u_{\lam_h}$ to be the solutions to the local problems \eqref{eq:local_1 hdg_wg}
  and \eqref{eq:local2 hdg_wg} (by replacing $\lam$ with $\lam_h$), respectively. Then it holds
  \begin{equation}
    u_h = u_{\lam_h} + u_f,
  \end{equation}
  and $\lam_h\in \mathbb M_{h,k}$ is the solution to the system \eqref{eq:sys}, i.e.
  \begin{equation*}
    d_h(\lam_h,\mu_h) = b_h(\mu_h)~~~\text{ for all $\mu_h \in \mathbb M_{h,k}$},
  \end{equation*}
  where
  \begin{eqnarray}
    d_h(\lam_h,\mu_h)\!\!&\!\!:=\!&\!\!\!(\bm A\bmnabla_w(u_{\lam_h,\lam_h}),\bmnabla_w(u_{\mu_h},\mu_h)) + \bndint{\alpha_T(P_T^{\partial}u_{\lam_h}-\lam_h),
      P_T^{\partial}u_{\mu_h}-\mu_h},\nonumber\\
    \\
    b_h(\mu_h)\!\!&\!\!:=\!\!&\!\!\!(f,u_{\mu_h}).
  \end{eqnarray}
\end{thm}
We consider two types of   WG methods (\cite{WangYe2013}):
\begin{itemize}
  \item {\bf Type 1.} $V(T) = P_k(T)$, $\bm W(T) = [P_k(T)]^d+P_K(T)\bm x$ and $\alpha_T = 0$;
  \item {\bf Type 2.}  $V(T) = P_{k-1}(T)~(k\geqslant 1)$, $\bm W(T) = [P_k(T)]^d$ and $\alpha_T = 0$.
\end{itemize}
In both cases, we can prove that {\bf Assumption \ref{ass:d_h}} holds.
\begin{thm}
  For {\bf Type 1} WG method, {\bf Assumption \ref{ass:d_h}} holds.
\end{thm}
\begin{proof}
  For $T\in\mathcal T_h$, define $\bmsigma:=\bmnabla_w(u_{\lam_h,\lam_h})|_T$. Then
  from \eqref{eq:local2 hdg_wg}  it follows
  $$
  (\bm A\bmsigma,\bmnabla_w^iv)_T = 0 \text{ for all $v\in V(T)$},
  $$
  which implies
  \begin{equation}\label{eq:987}
    \text{div}P_T^{rt}(\bm A\bmsigma) = 0,
  \end{equation}
  where $P_T^{tr}:[L^2(T)]^d\to\bm W(T)$ denotes the standard $L^2$-orthogonal projection operator.
  By the definition of $\bmnabla_w$, we have
  \begin{displaymath}
    \begin{split}
      (\bm A\bmsigma,\bmsigma)_T
      &=(P_T^{rt}(\bm A\bmsigma),\bmnabla_w(u_{\lam_h},\lam_h))_T\\
      &=-(\text{div}(P_T^{rt}(\bm A\bmsigma)),u_{\lam_h})_T+
      \bint{P_T^{rt}(\bm A\bmsigma)\cdot\bm n,\lam_h}{\partial T}\\
      &=\bint{P_T^{rt}(\bm A\bmsigma)\cdot\bm n,\lam_h-m_T(\lam_h)}{\partial T}
      ~~~~~~~~~\text{ (by \eqref{eq:987})}\\
      &\lesssim h_T^{-\frac{1}{2}}\norm{P_T^{rt}(\bm A\bmsigma)}_T\norm{\lam_h-m_T(\lam_h)}_{\partial T}\\
      &\lesssim\norm{\bm A\bmsigma}_T\vertiii{\lam_h}_{h,\partial T},
    \end{split}
  \end{displaymath}
  which shows immediately
  \begin{equation}\label{eq:988}
    (\bm A\bmsigma,\bmsigma)_T\lesssim\vertiii{\lam_h}^2_{h,\partial T}.
  \end{equation}

  On the other hand, for any $\bm\tau\in\bm W(T)$,   from the definition of $\bmnabla_w$ we have
  \begin{displaymath}
    \begin{split}
      (\bmsigma,\bm\tau)_T
      &=(\bmnabla_w(u_{\lam_h},\lam_h),\bm\tau)_T\\
      &=-(u_{\lam_h},\text{div}\bm\tau)_T+\bint{\bm\tau\cdot\bm n,\lam_h}{\partial T}\\
      &=(\bmnabla u_{\lam_h},\bm\tau)_T+\bint{\bm\tau\cdot\bm n,\lam_h-u_{\lam_h}}{\partial T},
    \end{split}
  \end{displaymath}
  which yields
  \begin{equation}\label{eq:989}
    (\bmsigma-\bmnabla u_{\lam_h},\bm\tau)_T = \bint{\bm\tau\cdot\bm n,\lam_h-u_{\lam_h}}{\partial T}.
  \end{equation}
  Taking $\bm\tau\in\bm W(T)$ in  \eqref{eq:989} with
  \begin{equation}
    \left\{
      \begin{array}{rcll}
        \int_F\bm\tau\cdot\bm nq &=& \int_F(\lam_h-u_{\lam_h})q&\text{ for all $F\in\mathcal F_T$ and $q\in P_k(F)$},\\
        \int_T\bm\tau\cdot\bmnabla v &=& 0 & \text{ for all $v\in V(T)$},
      \end{array}
    \right.
  \end{equation}
  we have
  \begin{displaymath}
    \begin{split}
      \norm{\lam_h-u_{\lam_h}}^2_{\partial T}
      &=(\bmsigma-\bmnabla u_{\lam_h},\bm\tau)_T
      =(\bmsigma,\bm\tau)_T
      \leqslant\norm{\bmsigma}_T\norm{\bm\tau}_T\\
      &\lesssim h_T^{\frac{1}{2}}\norm{\bmsigma}_T\norm{\lam_h-u_{\lam_h}}_{\partial T},
    \end{split}
  \end{displaymath}
  where we have used the estimate $\norm{\bm\tau}_T\lesssim h_T^{\frac{1}{2}}\norm{\lam_h-u_{\lam_h}}_{\partial T}$,
  which is a trivial result by applying the famous Piola mapping.
  The above inequality leads to
  \begin{equation}\label{eq:900}
    \norm{\lam_h-u_{\lam_h}}_{\partial T}\lesssim h_T^{\frac{1}{2}}\norm{\bmsigma}_T.
  \end{equation}

  Similarly,  taking $\bm\tau\in\bm W(T)$ in  \eqref{eq:989} with
  \begin{equation}
    \left\{
      \begin{array}{rcll}
        \int_F\bm\tau\cdot\bm nq &=& 0&\text{ for all $F\in\mathcal F_T$ and $q\in P_k(F)$},\\
        \int_T\bm\tau,\bmnabla v)_T &=& (\bmnabla u_{\lam_h},\bmnabla v)_T & \text{ for all $v\in V(T)$},
      \end{array}
    \right.
  \end{equation}
  we have
  \begin{equation}\label{eq:990}
    \vertiii{u_{\lam_h}}_{1,T}\lesssim\norm{\bmsigma}_T.
  \end{equation}
  Since by a standard scaling argument it holds
  \begin{equation}
    \norm{\lam_h-m_T(\lam_h)}_{\partial T}\sim\inf_{c\in\mathbb R}\norm{\lam_h-c}_{\partial T},
  \end{equation}
  we have
  \begin{displaymath}
    \begin{split}
      \norm{\lam_h-m_T(\lam_h)}_{\partial T}
      &\sim\inf_{c\in\mathbb R}\norm{\lam_h-c}_{\partial T}\\
      &\lesssim\norm{\lam_h-u_{\lam_h}}_{\partial T} + \inf_{c\in\mathbb R}\norm{u_{\lam_h}-c}_{\partial T}\\
      &\lesssim\norm{\lam_h-u_{\lam_h}}_{\partial T} + h_T^{\frac{1}{2}}\verti{u_{\lam_h}}_{1,T}\\
      &\lesssim h_T^{\frac{1}{2}}\norm{\bmsigma}_T, ~~~~\text{ (by \eqref{eq:900} and \eqref{eq:990})}
    \end{split}
  \end{displaymath}
  which, together with \eqref{eq:988}, yields
  \begin{equation}
    (\bm A\bmnabla_w(u_{\lam_h},\lam_h),\bmnabla_w(u_{\lam_h},\lam_h))_T\sim\verti{\lam_h}_{h,\partial T}.
  \end{equation}
  As a result, the desired estimate \eqref{eq:assum d_h} follows immediately. This completes the proof.
\end{proof}
\begin{rem}
  Similarly, we can show that {\bf Assumption \ref{ass:d_h}}
  holds for {\bf Type 2} WG method.
\end{rem}
\begin{rem}\label{equivalence-HDG-WG}
  If $\bm A$ is a piecewise constant matrix, the two WG methods are equivalent to the
  hybridized RT mixed element method and the hybridized BDM mixed element method,
  respectively. We refer to (Remark 2.1, \cite{Li;Xie;2014;WG}) for the details.
\end{rem}

\subsection{Nonconforming finite element method}
In this subsection we take Crouzeix-Raviart element method \cite{CR} as an example to show that the theory
in Section 4 also applies to nonconforming methods.

At first, we introduce the Crouzeix-Raviart finite element space $\mathcal L_h^{CR}$ as follows.
\begin{equation}
  \begin{split}
    \mathcal L_h^{CR}:=\{&v_h\in L^2(\Omega): v_h|_T\in P_1(T), \forall T\in\mathcal T_h,
    v_h\text{ is continuous at the }\\
    &\text{gravity point of each interior face of $\mathcal T_h$ and vanishes at the }\\
    &\text{gravity point of each face of $\mathcal T_h$ that lies on $\partial\Omega$}\}.
  \end{split}
\end{equation}
As we know, the standard discretization of the Crouzeix-Raviart element method reads as follows: Seek $u_h\in\mathcal L_h^{CR}$ such
that
\begin{equation}\label{eq_CR}
  (\bm A\bmnabla_h u_h,\bmnabla_h v_h) = (f, v_h)~~~\text{ for all $v_h\in \mathcal L_h^{CR}$},
\end{equation}
where $\bmnabla_h v_h$ is given by
\begin{displaymath}
  \bmnabla_h v_h|_T:=\bmnabla (v_h|_T)~~~\text{ for all $T\in\mathcal T_h$}.
\end{displaymath}

We define an operator $\widetilde\Pi_h:\mathcal L_h^{CR}\to M_{h,0}$ by
\begin{equation}
  \widetilde\Pi_h v_h|_F := \frac{1}{|F|}\int_Fv_h~~~\text{ for all $F\in\mathcal F_h$}.
\end{equation}
Obviously, $\widetilde\Pi_h$ is a bijective map, and its inverse map $\widetilde\Pi_h^{-1}: M_{h,0}\to\mathcal L_h^{CR}$
satisfies
\begin{equation}\label{421}
  \int_F\widetilde\Pi_h^{-1} \mu_h=\int_F\mu_h~~~\text{ for all $F\in\mathcal F_h$ and $\mu_h\in M_{h,0}$}.
\end{equation}
By denoting $\mu_h:=\widetilde\Pi_h v_h$ and $\lambda_h:=\widetilde\Pi_h u_h$, the system \eqref{eq_CR} is
equivalent to the system \eqref{eq:sys}, i.e.
\begin{equation*}
  d_h(\lam_h,\mu_h) = b_h(\mu_h)~~~\text{ for all $\mu_h \in \mathbb M_{h,0}$},
\end{equation*}
where
\begin{eqnarray}\label{CR_equi}
  d_h(\lam_h,\mu_h)\!\!&\!\!:=\!\! &\!\!(\bm A\bmnabla_h\widetilde\Pi_h^{-1}\lam_h,\bmnabla_h\widetilde\Pi_h^{-1}\mu_h),\\
  b_h(\mu_h)\!\! &\!\!:=&\!\! (f, \widetilde\Pi_h^{-1}\mu_h).\nonumber
\end{eqnarray}

By using  standard scaling arguments, it is easy to verify that {\bf Assumption \ref{ass:d_h}} holds in this case.
\begin{rem}\label {equivalence-HDG4-CR}
  We use $d_h^{hdg}$ and $d_h^{cr}$ to denote the bilinear forms defined in \eqref{def_d_h_hdg}
  and \eqref{CR_equi}, respectively. When $\bm A$ is a piecewise constant matrix, we can show
  that for {\bf Type 4} HDG method ($k=0$), $d_h^{hdg} = d_h^{cr}$. In fact, in this case we
  have, for any $T\in \mathcal T_h$,
  $$		M(\partial T):=\left\{\mu \in L^2(\partial T): \mu |_F \in P_0(F)\text{ for each face $F$ of $T$}\right\},
  $$
  $$V(T) = P_{1}(T),\quad \bm W(T) = [P_0(T)]^d,\quad \alpha_T = O(h_T^{-1}).$$
  From \eqref{eq:local_1 hdg_b} it follows
  \begin{displaymath}
    \langle\alpha_T(P_T^{\partial}(u_{\lam}-\lam),P_T^{\partial}v-\mu\rangle_{\partial T} = 0~~~\text{ for all $(v,\mu)\in V(T)\times M(\partial T)$}.
  \end{displaymath}
  Thus, in view of  \eqref{def_d_h_hdg}, it holds
  \begin{equation}\label{new_d_hnew}
    d_h^{hdg}(\lam_h,\mu_h) = (\bm C\bmsigma_{\lam_h},\bmsigma_{\mu_h}).
  \end{equation}
  On the other hand, it follows from $\eqref{eq:local_1 hdg_a}$ and \eqref{421} that
  $$
  (\bm C\bmsigma_{\mu_h},\bm\tau)_T=\bint{\mu_h,\bm\tau\cdot\bm n}{\partial T}=
  \bint{\widetilde\Pi_h^{-1} \mu_h,\bm\tau\cdot\bm n}{\partial T}=(\bmnabla\widetilde\Pi_h^{-1} \mu_h,\bm\tau)_T
  $$
  holds for all $\bm\tau\in \bm W(T)$. Since $\bm C=\bm A^{-1}$ is a constant matrix on $T$, the above equality means
  \begin{equation}\label{eq:139}
    \bmnabla_h\widetilde\Pi_h^{-1} \mu_h=\bm C\bmsigma_{\mu_h}.
  \end{equation}
  Thus, in light of  \eqref{new_d_hnew}-\eqref{eq:139} and  \eqref{CR_equi} we have
  \begin{displaymath}
    \begin{split}
      d_h^{hdg}(\lam_h,\mu_h)=(\bm A\bmnabla\widetilde\Pi_h^{-1}\lam_h,\bmnabla\widetilde\Pi_h^{-1}\mu_h)=d_h^{cr}(\lam_h,\mu_h).
    \end{split}
  \end{displaymath}
\end{rem}
\begin{rem} \label{equivalence-HDG1-CR}
  As shown in \cite{COCKBURN_2004,COCKBURN_2005}, when $\bm A$ is a piecewise constant matrix,
  the stiffness matrix of $d_h(\cdot,\cdot)$ arising from the lowest order hybridized RT mixed
  finite element method, i.e. {\bf Type 1} HDG method ($k=0$) in Subsection \ref{ssec:HDG},
  is the same as that arising from the Crouzeix-Raviart element method.
\end{rem}

\begin{rem}
  From Remarks \ref{equivalence-HDG-WG}-\ref{equivalence-HDG1-CR}, we know that when
  $\bm A$ is a piecewise constant matrix and $k=0$, the four methods, namely {\bf Type 1}
  and {\bf Type 4} HDG methods in Subsection \ref{ssec:HDG}, {\bf Type 1} WG method in
  Subsection \ref{ssec:WG}, and the Crouzeix-Raviart element method, lead to the same
  bilinear form $d_h(\cdot,\cdot)$, and hence   share the same optimal preconditioners.
\end{rem}

\section{Numerical experiments}\label{sec:numer}
In this section, we report several numerical examples in two-dimensions to verify the theoretical
results of Theorem \ref{thm_cond_BD} and Theorem \ref{thm:aux_pre}.  We only consider the problem
\eqref{eq:model} with the diffusion tensor $\bm A=\bm I$, where $\bm I$ is the identity matrix. We
test two types of HDG methods, i.e., {\bf Type 3} HDG method ($k=0,1$) with $\alpha_T = 1$ and
{\bf Type 4} HDG method ($k=0,1$) with $\alpha_T = h_T^{-1}$ for all $T\in\mathcal T_h$. We refer
to \cite{multigrid_hdg, Li;Xie;2015;HDG}  for more numerical results of {\bf Type 3} HDG method.

\noindent {\bf Example 1}. We set $\Omega = (0,1)\times (0,1)$ (a square domain) with the initial triangulation
$\mathcal T_0$ (Figure \ref{fig:first}). We produce a sequence of triangulations
$\{\mathcal T_j:j=1,2,\cdots, 10\}$ by a successive refinement procedure: connecting the midpoints of
three edges of each triangle.

For each $j=5,6,\cdots, 10$,  we set $\mathcal T_h=\mathcal T_j$, and let $\mathcal D_h$
and $\mathcal B_h$ be defined by \eqref{def_D_h} and \eqref{def_B_h} respectively.
Suppose we are to solve the
system $\mathcal D_hx=b_h$, where $b_h$ is a zero vector. Taking $x_0 = (1,1,\dots,1)^t$ as the
initial value, we  use the famous preconditioned conjugate gradient
method (PCG) to solve this system with  the preconditioner $\mathcal B_h$.  The stopping criterion is that the initial error, i.e. $\sqrt{x_0^T\mathcal D_hx_0}$,
is reduced by a factor of $10^{-6}$.

In the case $k=0$, the prolongation operators $\Pi_h^1$ and $\Pi_h^2$ are equivalent,
so we have one BPX preconditioner. In the case $k=1$,  we
have two different BPX preconditioners since $\Pi_h^1$ and $\Pi_h^2$ are not equivalent, and we compute both cases. The
corresponding numerical results, i.e. the number of iterations in PCG, are listed in Table \ref{tab:first}.

\noindent {\bf Example 2}.  The only difference between this example and Example 1  is that we set
$\Omega=(0,1)\times(0,1)/[0,1)\times\{0.5\}$ (a crack domain) with the initial triangulation
$\mathcal T_0$ (Figure \ref{fig:second}). The corresponding numerical
results are presented in Table \ref{tab:second}. 

\noindent {\bf Example 3}. This example  is to verify Theorem \ref{thm:aux_pre} for graded triangulations.
We only consider {\bf Type 3} HDG method. For simplicity we  only highlight the difference between this example and
the previous two examples. 

We set $\Omega = (-1,1)\times(-1,1)$ and define
$\bm A(x,y)=\text{diag}(a(x,y), a(x,y))$ with
\begin{displaymath}
  a(x,y):=\left\{
    \begin{array}{rrr}
      1, & -1 < x < 0, & -1 < y < 0;\\
      7, &  0 < x < 1, & -1 < y < 0;\\
      17,&  0 < x < 1, &  0 < y < 1;\\
      3, & -1 < x < 0, & 0 < y < 1.
    \end{array}
  \right.
\end{displaymath}
We show the first two triangulations $\mathcal T_0$ and $\mathcal T_1$ in Figure \ref{fig:third}
and produce a sequence of graded triangulations $\{\mathcal T_j:j=0,1,\cdots,25\}$ in a successive way:
$\mathcal T_{j+1}$ ($j=2,3,\cdots,24$) is obtained by refining the smallest square containing the
origin in $\mathcal T_j$ (in $\mathcal T_1$, the vertexes of the square to refine is in red color)
as what has been done from $\mathcal T_0$ to $\mathcal T_1$.   $\mathcal T_{25}$ is shown in Figure
\ref{fig:fourth}. For each $j=5,10,15, 20, 25$, we set $\mathcal T_h=\mathcal T_j$ and, in the
definition \eqref{eq:def B^G_h} of $B^G_h$, we set $S_h$ to be the standard symmetric Gauss-Seidel
iteration and set $\we{B_h}=A_h^{-1}$, where $A_h^{-1}:V_h^c\to V_h^c$ is defined by
$$
(A_hu_h,v_h):=(\bm A\bmnabla u_h,\bmnabla v_h)~\text{ for all $u_h,v_h\in V_h^c$}.
$$
The corresponding numerical results are presented in Table \ref{tab:third experiment}.
\begin{figure}[H]
  \begin{floatrow}
    \ffigbox
    {\includegraphics[width=0.35\linewidth]{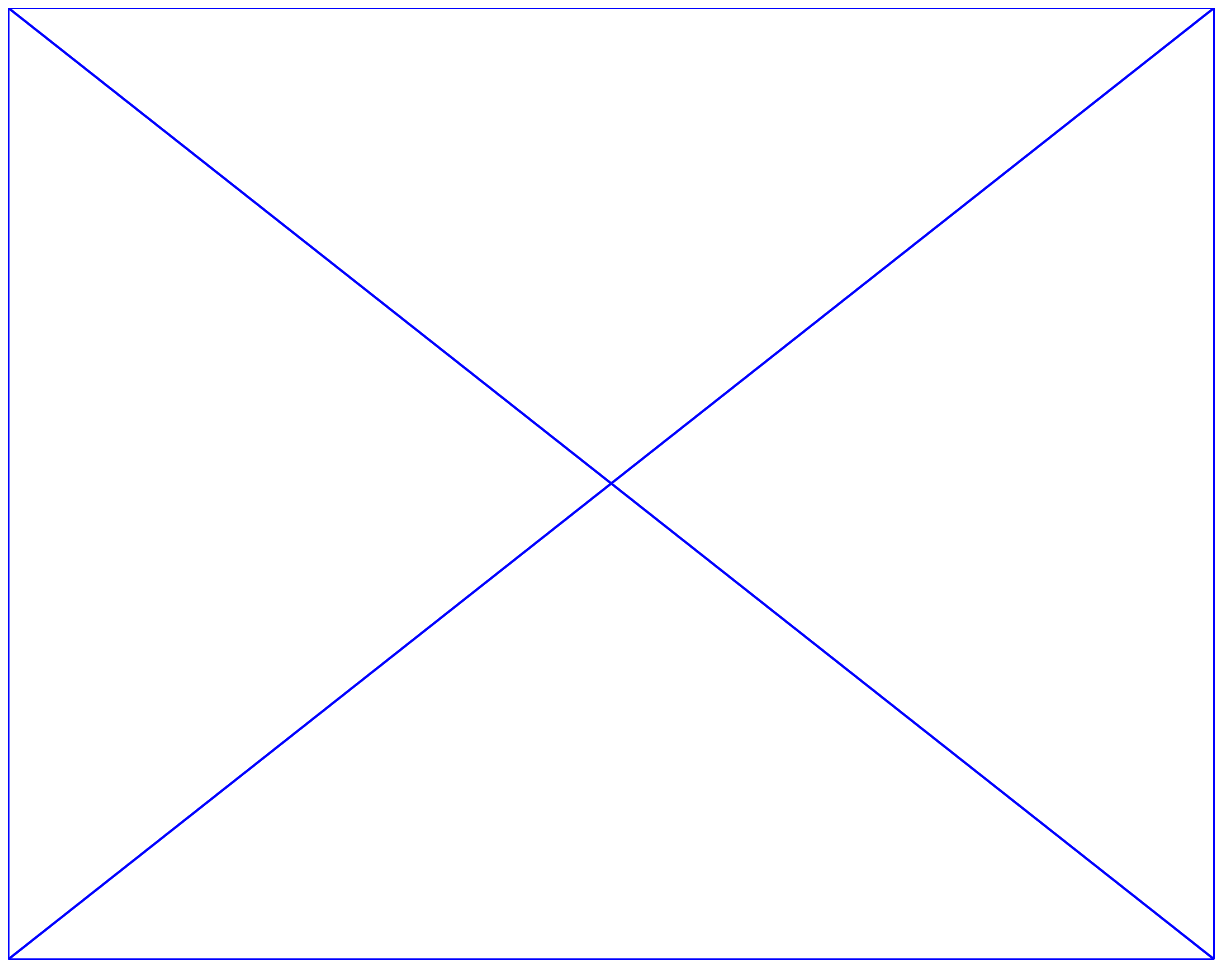}
      \includegraphics[width=0.35\linewidth]{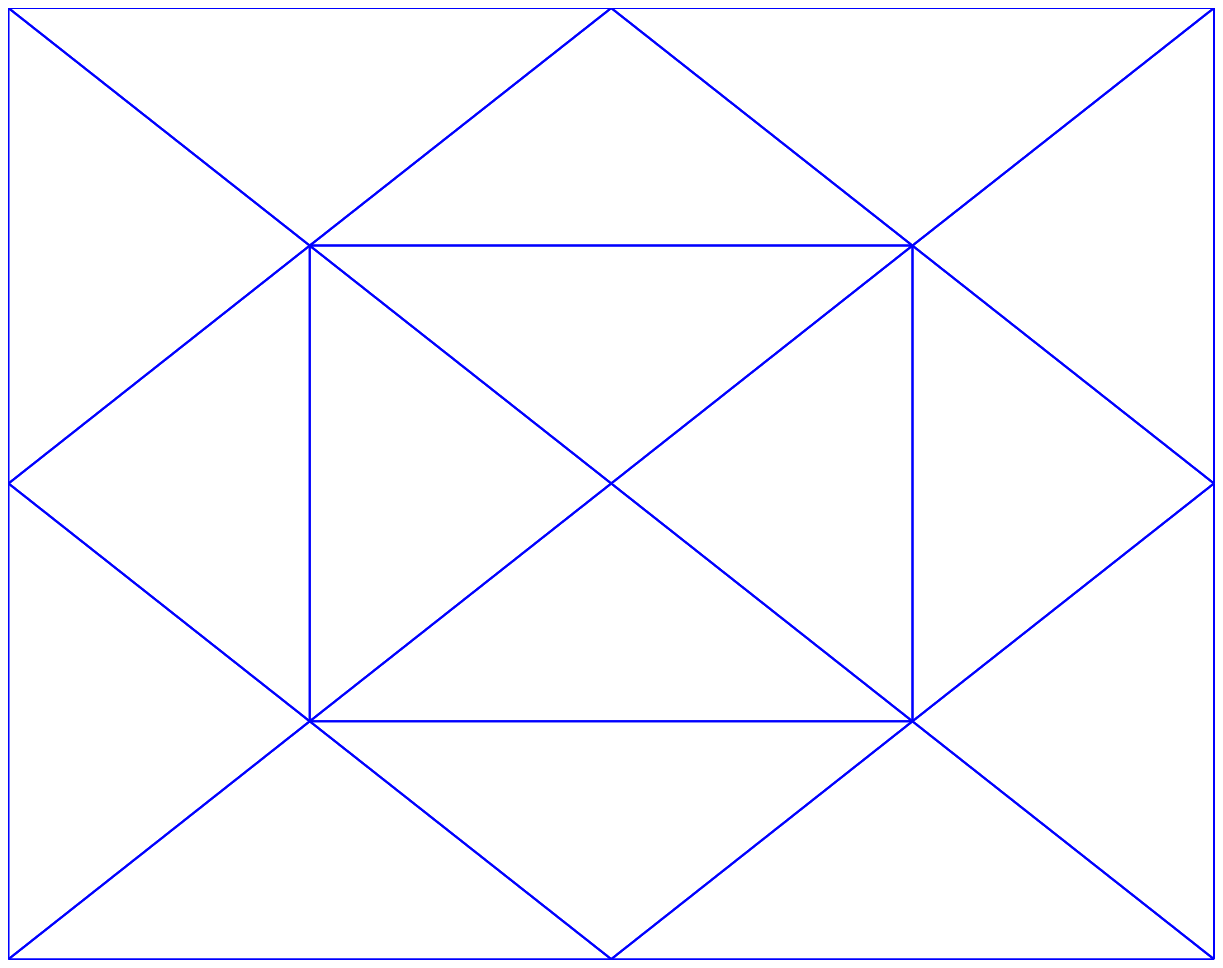}}
    {
      \caption{$\mathcal T_0$ (left) and $\mathcal T_1$ (right) on square domain}\label{fig:first}
    }
    \ffigbox
    {
      \includegraphics[width=0.35\linewidth]{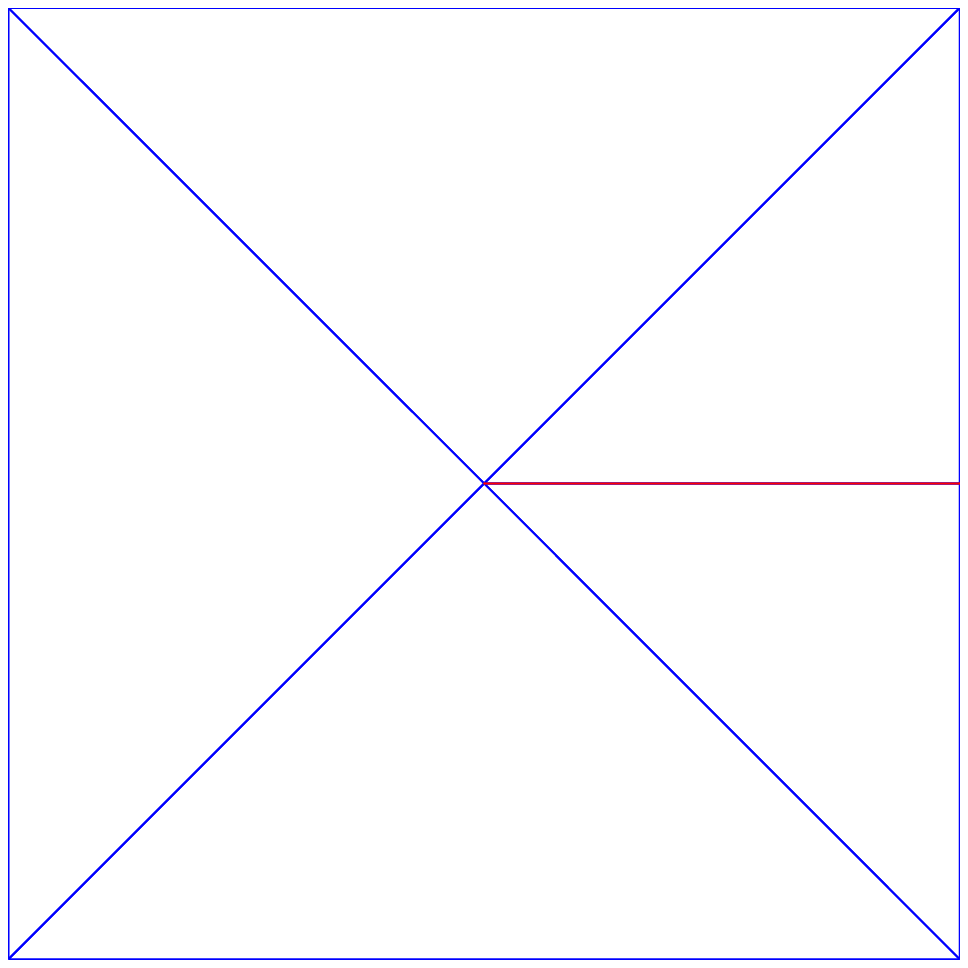}
      \includegraphics[width=0.35\linewidth]{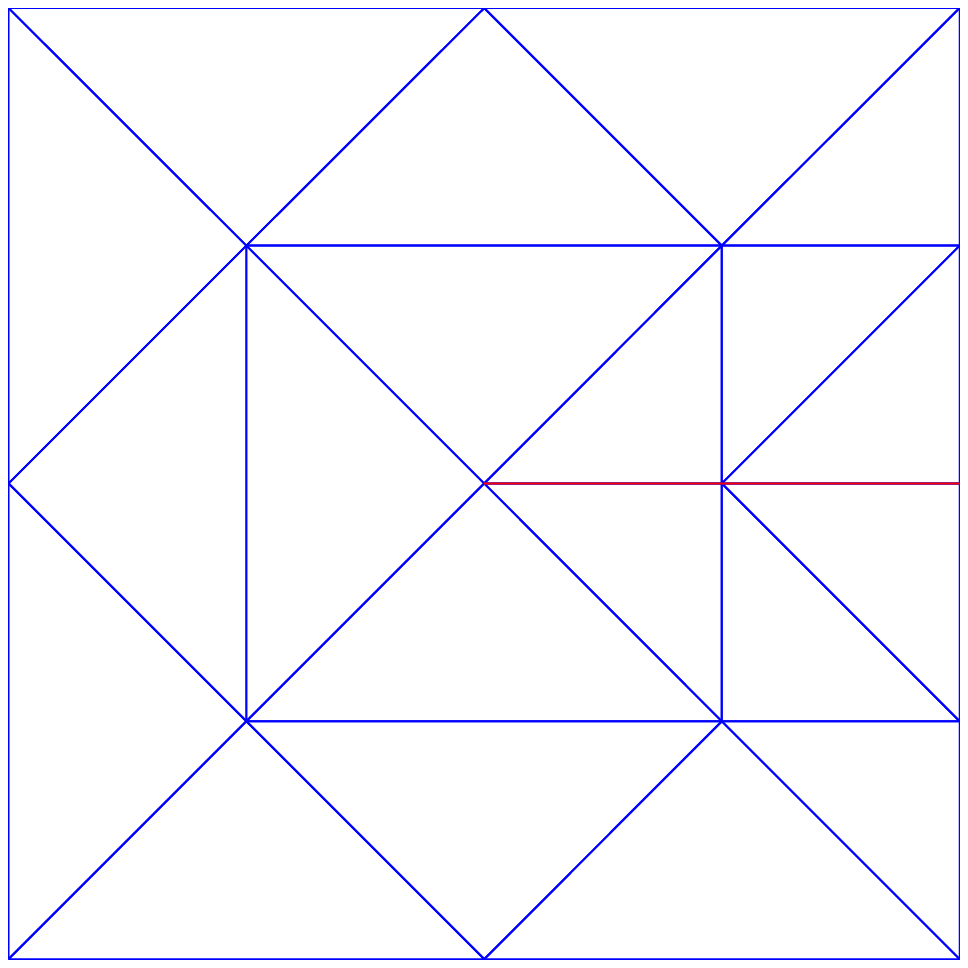}
    }
    {
      \caption{$\mathcal T_0$ (left) and $\mathcal T_1$ (right) on crack domain}\label{fig:second}
    }
  \end{floatrow}
\end{figure}
\begin{figure}[H]
  \begin{floatrow}
    \ffigbox
    {
      \includegraphics[width=0.3\linewidth]{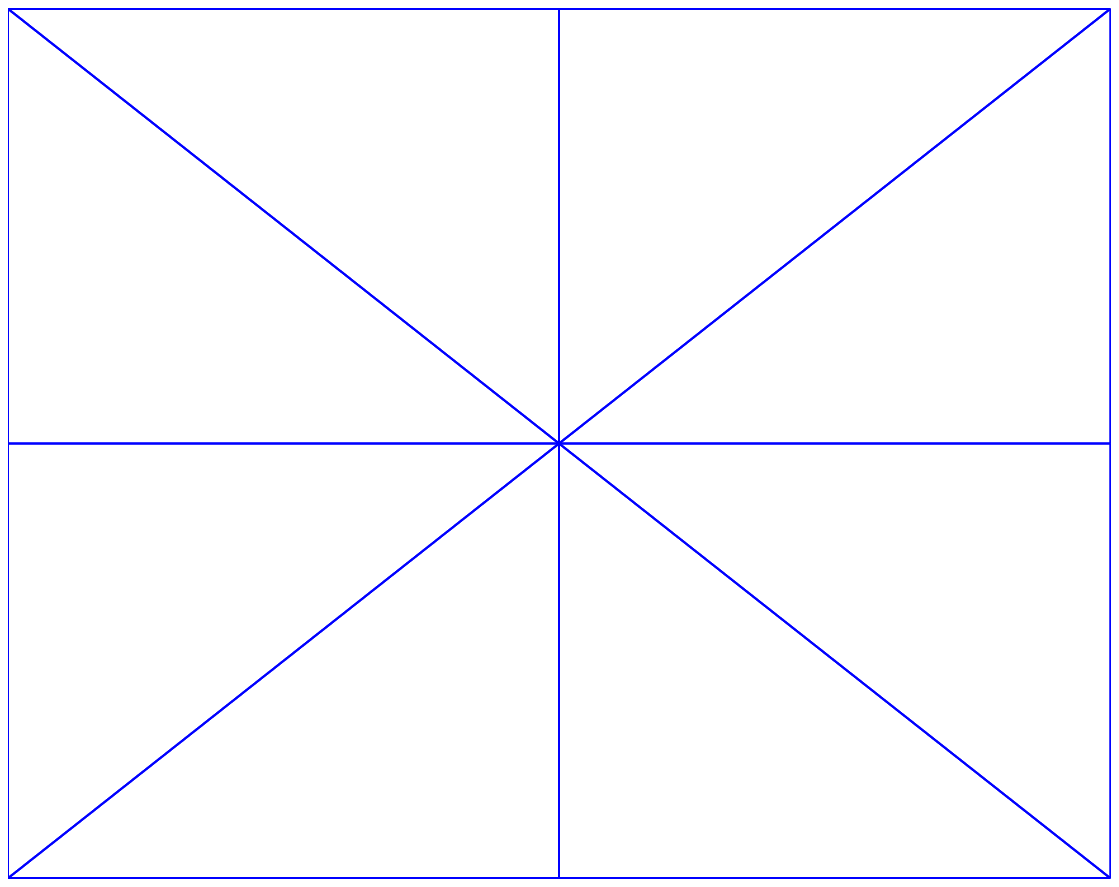}
      \includegraphics[width=0.3\linewidth]{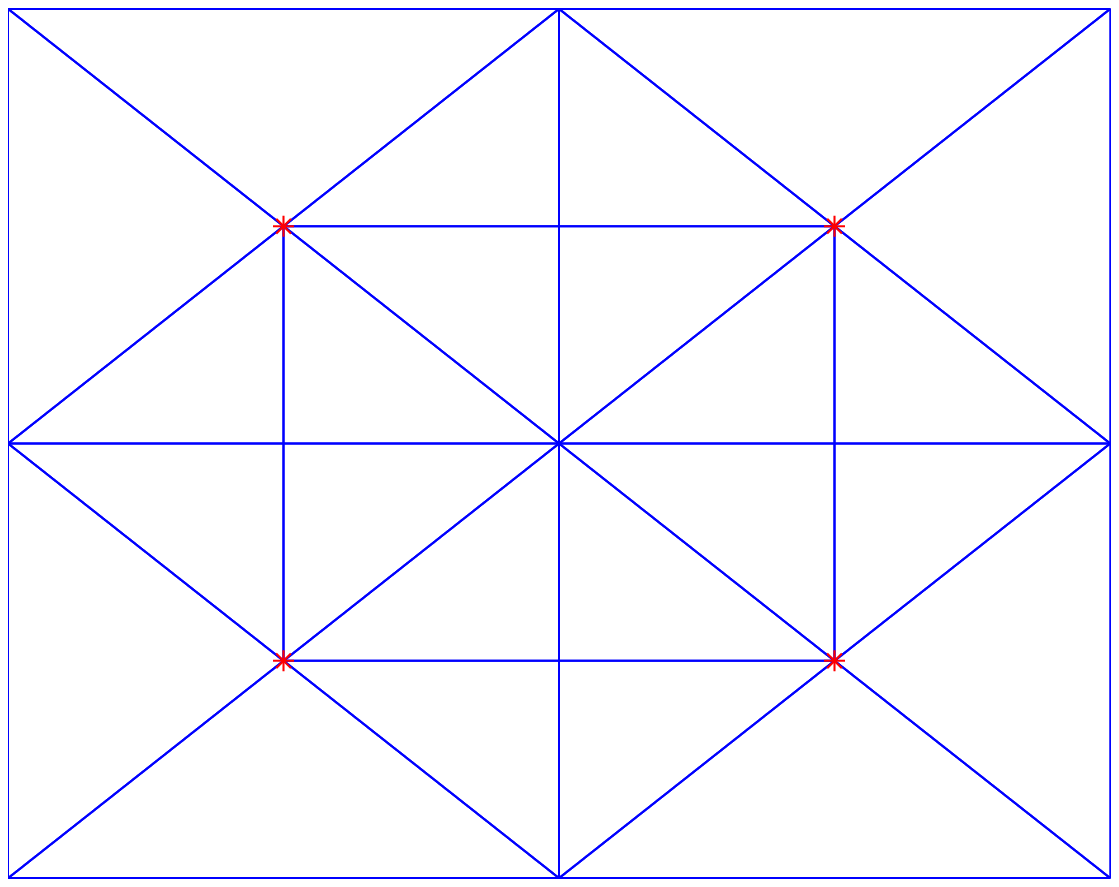}
    }
    {
      \caption{$\mathcal T_0$ (left) and $\mathcal T_1$ (right)}\label{fig:third}
    }
    \ffigbox
    {
      \includegraphics[width=0.3\linewidth]{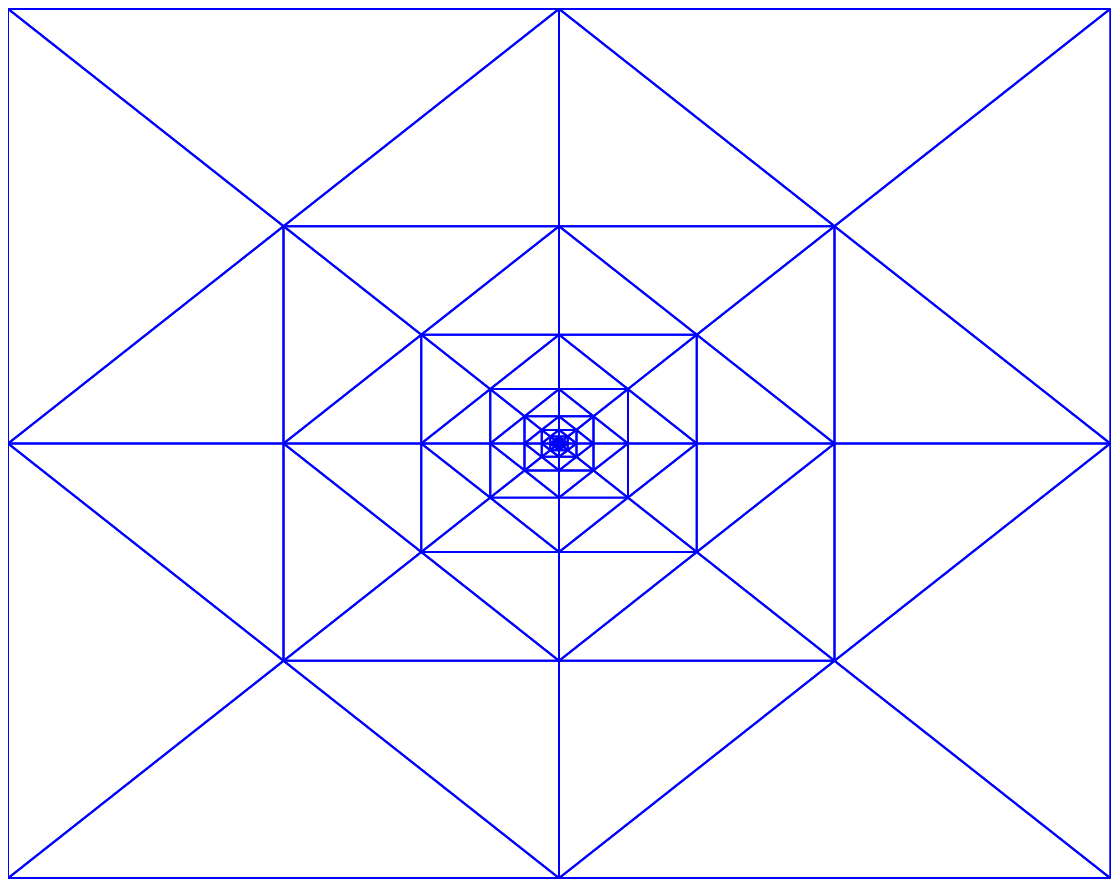}
    }
    {
      \caption{$\mathcal T_{25}$}\label{fig:fourth}
    }
  \end{floatrow}
\end{figure}
\par
\begin{table}[H]\tiny
  \begin{center}
    \caption{Numerical results for the Example 1 (dof denotes the number of degrees of freedom)}\label{tab:first}
    \begin{tabular}{p{1.pt}p{1.pt} p{3.pt} p{4.pt} p{5.pt} p{6.pt} p{8.pt} p{14.pt}}
      \toprule
      $k$ &    & $\mathcal T_5$ & $\mathcal T_6$ & $\mathcal T_7$ & $\mathcal T_8$ & $\mathcal T_9$ & $\mathcal T_{10}$ \\
      \midrule
      0 & \text{dof}   & 6080   & 24448   & 98048   & 392704 & 1571840 & 6289408  \\
      &      & 20    & 20      & 20      & 20     & 20      & 19\\
      \midrule
      1 & \text{dof}   & 12288  &49152 & 196608 & 786432 & 3145728     & 12582912     \\
      &  $\Pi_h^1$  & 54     &  57  & 59     &  60    &  61 &  62    \\
      &  $\Pi_h^2$  & 31     &  32  & 33     &  34    & 34  &  35  \\
      \bottomrule
      &&&\text{{\bf Type 3} HDG method}
    \end{tabular}
    \begin{tabular}{p{1.pt}p{1.pt} p{3.pt} p{4.pt} p{5.pt} p{6.pt} p{8.pt} p{14.pt}}
      \toprule
      $k$ &    & $\mathcal T_5$ & $\mathcal T_6$ & $\mathcal T_7$ & $\mathcal T_8$ & $\mathcal T_9$ & $\mathcal T_{10}$ \\
      \midrule
      0 & \text{dof}   & 6080   & 24448   & 98048   & 392704 & 1571840 & 6289408  \\
      &      & 20    & 20      & 20      & 20     & 20      & 19\\
      \midrule
      1 & \text{dof}   & 12288  &49152 & 196608 & 786432 & 3145728  & 12582912     \\
      &  $\Pi_h^1$ & 54     &  57  & 59     &  60    &  61      &  62    \\
      &  $\Pi_h^2$ & 31     &  32  & 33     &  34    &  34      &  35  \\
      \bottomrule
      &&&\text{{\bf Type 4} HDG method}
    \end{tabular}
  \end{center}
\end{table}
\begin{table}[H]\tiny
  \begin{center}
    \caption{Numerical results for Example 2}\label{tab:second}
    \begin{tabular}{p{1.pt} p{1.pt} p{3.pt} p{4.pt} p{5.pt} p{6.pt} p{8.pt} p{14.pt}}
      \toprule
      $k$ & & $\mathcal T_5$ & $\mathcal T_6$ & $\mathcal T_7$ & $\mathcal T_8$ & $\mathcal T_9$ & $\mathcal T_{10}$ \\
      \midrule
      0 & \text{dof} & 7568   & 30496   & 122432   & 490624   & 1964288 & 7860732  \\
      &      & 26 &  26 & 27 & 27 & 27 & 27\\
      \midrule
      1 & \text{dof} & 15360 & 61440  & 245760 & 983040 & 3932160  & 15728640   \\
      & $\Pi_h^1$  & 55    & 58     & 59     & 61     & 62  & 63   \\
      & $\Pi_h^2$  & 33    & 34     & 35     & 36     & 36  & 37   \\
      \bottomrule
      &&&\text{{\bf Type 3} HDG method}
    \end{tabular}
    \begin{tabular}{p{1.pt} p{1.pt} p{3.pt} p{4.pt} p{5.pt} p{6.pt} p{8.pt} p{14.pt}}
      \toprule
      $k$ & & $\mathcal T_5$ & $\mathcal T_6$ & $\mathcal T_7$ & $\mathcal T_8$ & $\mathcal T_9$ & $\mathcal T_{10}$ \\
      \midrule
      0 & \text{dof} & 7568   & 30496   & 122432   & 490624   & 1964288 & 7860732  \\
      &      & 26 &  27 & 27 & 27 & 27 & 27\\
      \midrule
      1 & \text{dof} & 15360 & 61440  & 245760 & 983040 & 3932160  & 15728640   \\
      & $\Pi_h^1$  & 55    & 58     & 59     & 61     & 62  & 63   \\
      & $\Pi_h^2$  & 33    & 34     & 35     & 36     & 36  & 37   \\
      \bottomrule
      &&&\text{{\bf Type 4} HDG method}
    \end{tabular}
  \end{center}
\end{table}

\begin{table}[H]\tiny
  \caption{Numerical results for Example 3}\label{tab:third experiment}
  \begin{tabular}{p{4.pt} p{3.pt}  p{10.pt} p{10.pt} p{10.pt} p{10.pt} p{10.pt} p{10.pt}}
    \toprule
    $k$&  & $\mathcal T_5$ & $\mathcal T_{10}$ & $\mathcal T_{15}$ & $\mathcal T_{20}$ & $\mathcal T_{25}$\\
    \midrule
    0  & &  14 & 14 & 14 & 14 & 14\\
    \midrule
    1  & $\Pi_h^1$&  26  & 26   &  27  &  26   & 26\\
    1  & $\Pi_h^2$ &  18  & 18 & 18 & 18 & 18\\
    \bottomrule
  \end{tabular}
\end{table}

From Tables \ref{tab:first}-\ref{tab:third experiment} we have the following observations.
\begin{itemize}
  \item  For all the examples,  the numbers of iterations in PCG are independent of the mesh size. This means the proposed preconditioners are optimal.  Besides,  the prolongation operator $\Pi_h^2$ behaves  better than   $\Pi_h^1$  in the case $k=1$.
  \item  Example 1 admits the full elliptic regularity, while Example 2
    only admits the regularity estimate $\norm{u}_{1+\alpha,\Omega}\leqslant C_{\alpha,\Omega}
    \norm{f}_{\alpha-1,\Omega}$ with $\alpha\in (0,\frac{1}{2})$.  These two examples  confirm  that the proposed BPX preconditioner is optimal.  This is conformable to Theorem \ref{thm_cond_BD}.
  \item  Example 3 confirms Theorem \ref{thm:aux_pre}, where the triangulation $\mathcal T_h$ is not quasi-uniform.
\end{itemize}

\end{document}